\documentclass[11pt,reqno]{amsart}

\usepackage{amsmath, amsfonts, amsthm, amssymb, graphicx, amscd,color,colortbl,}
\usepackage{float,epsf,subfigure}
\usepackage[english]{babel}

\numberwithin{equation}{section} \theoremstyle{plain}
\newtheorem{theorem}{Theorem}[section]
\newtheorem{lemma}[theorem]{Lemma}
\newtheorem{proposition}[theorem]{Proposition}

\theoremstyle{definition}

\theoremstyle{remark}
\newtheorem{remark}[theorem]{Remark}

\theoremstyle{example}

\renewcommand{\a}{\alpha}

\renewcommand{\r}{\rho}

\newcommand{\n}{\mathbf{n}}

\def\bes{\begin{equation*}}
\def\ees{\end{equation*}}

\newcommand{\ra}{\rightarrow}

\newcommand{\be}{\begin{equation}}
\newcommand{\ee}{\end{equation}}
\newcommand{\bee}{\begin{eqnarray}}
\newcommand{\eee}{\end{eqnarray}}

\begin{document}
\title[Global Steady Subsonic Euler Flows through Infinitely Long Nozzles]{Global Steady
Subsonic Flows through Infinitely Long Nozzles for the Full Euler Equations}
\author{Gui-Qiang Chen $\qquad$ Xuemei Deng $\qquad$ Wei Xiang}

\address{Gui-Qiang G. Chen, School of Mathematical Sciences, Fudan University,
 Shanghai 200433, China; Mathematical Institute, University of Oxford,
         Oxford, OX1 3LB, UK; Department of Mathematics, Northwestern University,
         Evanston, IL 60208, USA}
 \email{\tt chengq@maths.ox.ac.uk}

\address{Xuemei Deng, College of Science, China Three Gorges University, Yichang, Hubei 443002, China;
         Mathematical Institute, University of Oxford,
         Oxford, OX1 3LB, UK; School of Mathematical Sciences, Xiamen University,
         Xiamen, Fujian 361005, China}
\email{\tt dmeimeisx@yahoo.com.cn}

\address{Wei Xiang, Mathematical Institute, University of Oxford,
         Oxford, OX1 3LB, UK; and School of Mathematical Sciences, Fudan University,
 Shanghai 200433, China}
         \email{\tt xiangwei0818@gmail.com}



\date{\today}
\maketitle

\begin{abstract}
We are concerned with global steady subsonic flows through general infinitely long nozzles
for the full Euler equations. The problem is formulated as a boundary value problem in the unbounded domain
for a nonlinear elliptic equation of second order in terms of the stream function.
It is established that, when the oscillation of the entropy and Bernoulli functions at the upstream
is sufficiently small in $C^{1,1}$ and the mass flux is in a suitable regime,
there exists a unique global subsonic solution in a suitable class of general nozzles.
The assumptions are required to prevent from the occurrence of supersonic bubbles inside the nozzles.
The asymptotic behavior of subsonic flows at the downstream and upstream,
as well as the critical mass flux, has been clarified.
\end{abstract}

\section{Introduction}
We are concerned with global steady subsonic flows through general infinitely long nozzles for
the full Euler equations (without the isentropic and irrotational requirement).
The two-dimensional steady full Euler equations take the following form:
\bee
&&(\rho u)_{x_{1}}+(\rho v)_{x_{2}}=0,\label{1.1}\\
&&(\rho u^{2})_{x_{1}}+(\rho uv)_{x_{2}}+p_{x_{1}}=0,\label{1.2}\\
&&(\rho uv)_{x_{1}}+(\rho v^{2})_{x_{2}}+p_{x_{2}}=0,\label{1.3a}\\
&&(\rho u(E+\frac{p}{\rho}))_{x_{1}}+(\rho v(E+\frac{p}{\rho}))_{x_{2}}=0,\label{1.4a}
\eee
where $\rho$, $(u,v)$, $p$, and $E$ denote the density, velocity, pressure, and total energy respectively. Moreover,
\be
E=\frac{1}{2}(u^{2}+v^{2})+\frac{p}{(\gamma-1)\rho}
\ee
with adiabatic exponent $\gamma>1$.

Consider flows through an infinitely long nozzle given by
$$
\Omega=\{(x_{1},x_{2})\, :\, f_{1}(x_{1})<x_{2}<f_{2}(x_{1}),-\infty<x_{1}<\infty\},
$$
with the nozzle wall $\partial\Omega:=W_1\cup W_2$,
where
$$
W_{i}=\{(x_{1},x_{2})\,:\, x_{2}=f_{i}(x_{1}),-\infty<x_{1}<\infty\}, \quad i=1,2,
$$
as in Fig \ref{n}.

\bigskip

\begin{figure}[h!]
  \centering
      \includegraphics[width=0.62\textwidth]{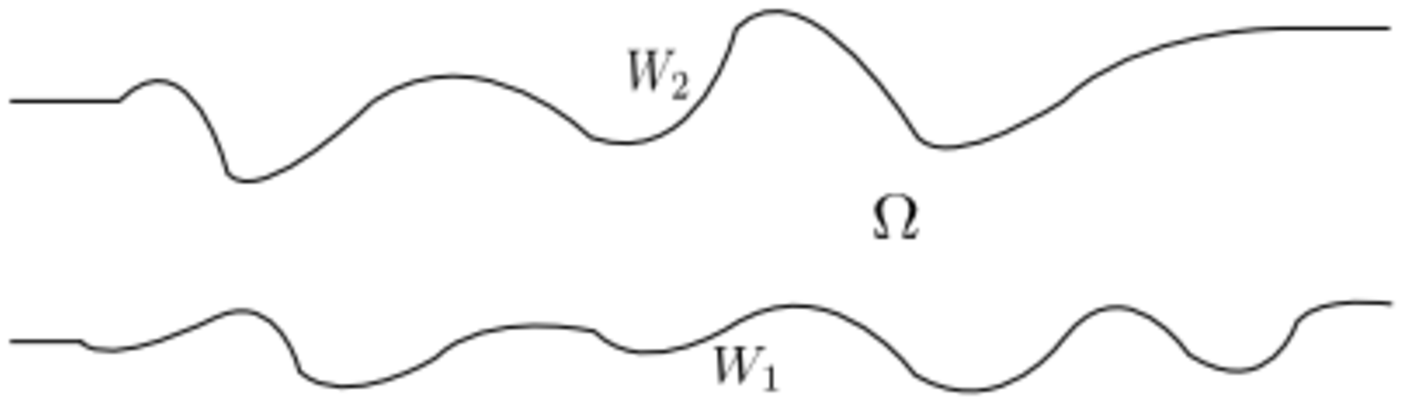}
  \caption{Infinite nozzle}
\label{n}
\end{figure}

Suppose that
$W_1$ and $W_2$ satisfy
\bee
&&f_{2}(x_1)>f_{1}(x_1) \qquad\qquad\qquad\,\,\,  \text{for}\  x_{1}\in(-\infty,\infty),\label{1.3}\\
&&f_{1}(x_1)\rightarrow 0,\  f_{2}(x_1)\rightarrow 1 \qquad\quad \text{as}\  x_{1}\rightarrow-\infty\quad\,\,  \text{in } C^{2,\alpha},\label{1.4}\\
&&f_{1}(x_1)\rightarrow a,\ f_{2}(x_1)\rightarrow b >a \quad\,\,
\text{as}\ x_{1}\rightarrow\infty\qquad\text{in } C^{2,\alpha}, \eee
and there exists $\alpha>0$ such that \be\label{1.6}
\|f_{i}\|_{C^{2,\alpha}(\mathbb{R})}\leq C, \qquad i=1,2, \ee for
some positive constant  $C$. It follows that $\Omega$ satisfies the
uniform exterior sphere condition with some uniform radius $r>0$.

Suppose that the nozzle walls are solid so that the flow satisfies the slip boundary condition:
\be\label{1.7}
(u,v)\cdot\n=0\qquad \text{on}\ \partial\Omega,
\ee
where $\n$ is the unit outward normal to the nozzle wall $\partial \Omega$.
It follows from  \eqref{1.1} and \eqref{1.7} that
\be\label{1.8}
\int_{\ell}\, (\rho u,\rho v)\cdot \n \, dl\equiv m
\ee
for some constant $m$, where $\ell$ is any curve transversal to the $x_{1}-$direction,
and $\n$ is the normal of $\ell$ in the positive $x_{1}-$direction.

If the flow is away from the vacuum state, it follows from \eqref{1.2}--\eqref{1.4a} that
\be\label{1.9}
(u,v)\cdot\nabla(\ln p-\gamma\ln \rho)=0,
\ee
which implies that $\frac{p}{\rho^{\gamma}}$ is a constant along each streamline,
provided that the solution is $C^{1}$--smooth.
We assume that the entropy function is given in the upstream, i.e.,
\be\label{1.10}
\frac{\gamma p}{(\gamma-1)\rho^{\gamma}}\rightarrow S(x_{2})\qquad\,\,
\text{as}\ x_{1}\rightarrow-\infty,
\ee
where $S(x_{2})$, defined on $[0,1]$, is the entropy.
The sonic speed of the flow is defined by
\begin{equation}\label{1.10a}
c=\sqrt{\frac{\gamma p}{\rho}}.
\end{equation}
By \eqref{1.1} and \eqref{1.4a}, we obtain
\be\label{1.14a}
(u,v)\cdot\nabla\big(\frac{1}{2}(u^{2}+v^{2})+\frac{\gamma p}{(\gamma-1)\rho}\big)=0.
\ee
This implies that $\frac{1}{2}(u^{2}+v^{2})+\frac{\gamma p}{(\gamma-1)\rho}$,
which is called the Bernoulli function, is a constant along each streamline.
We assume that the Bernoulli function is given in the upstream, i.e.,
\be\label{1.12}
\frac{(u^{2}+v^{2})}{2}+\frac{\gamma p}{(\gamma-1)\rho}\rightarrow B(x_{2})\qquad\,\,\,
\text{as}\ x_{1}\rightarrow-\infty,
\ee
where $B(x_2)$ is defined on $[0,1]$.

\bigskip
{\bf Problem 1}. {\it Solve the full Euler system \eqref{1.1}--\eqref{1.4a} with
the boundary condition \eqref{1.7},
the mass flux condition \eqref{1.8}, and the asymptotic
conditions \eqref{1.10} and \eqref{1.12}. }

\bigskip

Set
$$
\underline{S}= \inf\limits_{x_2\in[0,1]}S(x_2),
\qquad
\underline{B}= \inf\limits_{x_2\in[0,1]}B(x_2).
$$

The main results of this paper are the following.

\begin{theorem}[Main Theorem] Let the nozzle walls $\partial\Omega$
satisfy \eqref{1.3}--\eqref{1.6}, and let $\underline{S}>0$ and $\underline{B}>0$.
Then there exists $\delta_0>0$ such that, if
\be\label{1.13}
\|(S-\underline{S}, B-\underline{B})\|_{C^{1,1}([0,1])}\leq\delta
\qquad \mbox{for $0<\delta\leq \delta_0$},
\ee
and
\be\label{1.16a}
(SB^{-\gamma})'(0)\geq0, \qquad  (SB^{-\gamma})'(1)\leq0,
\ee
there exists $\hat{m}\geq2\delta^{1/8}_0$ such that, for any $m\in(\delta^{1/4},\hat{m})$,
there exists a global solution (i.e. a full Euler flow) $(\rho,u,v,p)\in C^{1,\alpha}(\bar{\Omega})$
of {\rm Problem 1} such that
\begin{itemize}
\item[(i)] Subsonicity and positivity of the horizontal velocity:
The flow is uniformly subsonic globally with positive horizontal velocity in the whole nozzle, i.e.,
\be\label{1.16}
\sup\limits_{\overline{\Omega}}(u^2+v^2-c^{2})<0, \qquad \,\, u>0\quad\text{in}\ \overline{\Omega};
\ee

\item[(ii)] Far field behavior:  The flow satisfies
the following asymptotic behavior in the far fields:

\smallskip
\noindent
{\rm (a)}  As $x_1\to -\infty$,
\begin{eqnarray}
&& p\to p_{0}>0,\quad u\to u_{0}(x_{2})>0,\quad
 (v, \rho)\to (0, \rho_0(x_2; p_0)), \label{1.18} \\
&& \nabla p\rightarrow0,\ \  \nabla u\rightarrow(0,
u_{0}'(x_{2})), \ \  \nabla v\rightarrow0, \ \  \nabla \rho\to (0, \rho_0'(x_2;p_0))
\quad \label{1.19}
\end{eqnarray}
uniformly for $x_{2}\in K_1\Subset(0,1)$;

\smallskip
\noindent
{\rm (b)} As $x_1\to \infty$,
\begin{eqnarray}
&& p\to p_{1}>0,\quad u\to u_{1}(x_{2})>0,\quad
 (v, \rho)\to (0, \rho_1(x_2;p_1)), \label{1.20}\\
&& \nabla p\rightarrow0,\ \  \nabla u\rightarrow(0,
u_{1}'(x_{2})), \ \  \nabla v\rightarrow0, \ \  \nabla \rho\to (0, \rho_1'(x_2;p_1))
\quad \label{1.21}
\end{eqnarray}
uniformly for  $x_{2}\in K_2\Subset(a,b)$,
where $p_0$ and $p_1$ are both positive constants,
$$
\rho_0(x_2;p_0)=\Big(\frac{\gamma p_0}{(\gamma-1)S(x_2)}\Big)^{\frac{1}{\gamma}},
\qquad
\rho_1(x_2;p_1)= \Big(\frac{\gamma p_1}{(\gamma-1)S(x_2)}\Big)^{\frac{1}{\gamma}},
$$
and
$p_0$, $p_1$, $u_0(x_2)$ and $u_1(x_2)$ can be determined by $m$, $S(x_2)$,
$B(x_2)$, and $b-a$ uniquely;

\item[(iii)] Uniqueness: The full Euler flow of {\rm Problem 1}
satisfying \eqref{1.16} and the asymptotic behavior
\eqref{1.18}--\eqref{1.21} is unique.

\item[(iv)] Critical mass flux:
$\hat{m}$ is the upper critical mass flux for the existence of subsonic flow in the following sense: either
\be\label{1.25}
\sup\limits_{\overline{\Omega}}(u^2+v^2-c^{2})\rightarrow0\qquad\text{as}\ m\rightarrow\hat{m},
\ee
or there is no $\sigma>0$ such that, for all $m\in(\hat{m},\hat{m}+\sigma)$,
there are full Euler flows of {\rm Problem 1}
satisfying
\eqref{1.16},
the asymptotic behavior \eqref{1.18}--\eqref{1.21}, and
\be\label{1.24}
\sup\limits_{m\in(\hat{m},\hat{m}+\sigma)}\sup\limits_{\overline{\Omega}}(c^{2}-(u^2+v^2))>0.
\ee
\end{itemize}
\end{theorem}

The assumptions in Theorem 1.1 are required to prevent from the occurrence of supersonic
bubbles inside the infinitely long nozzles.

There has been some literature on the analysis of the infinitely nozzle problems.
For potential flows, Chen-Feldman \cite{ChenFeldman0,ChenFeldman} established the existence
and stability of multidimensional
transonic flows through an infinite nozzle of arbitrary cross-sections;
also see Chen-Dafermos-Slemrod-Wang \cite{CDSW} and Kim \cite{Kim}.
Xie-Xin \cite{xx1} established the existence of global subsonic isentropic flows and obtained
the critical upper bound of mass flux under the assumption that the derivative of
the Bernoulli function equals to zero on the two boundaries.
For the steady full Euler equations,
Chen-Chen-Feldmann \cite{CCF} established the first existence of global transonic flows
in two-dimensional infinite nozzles of slowly varying cross-sections;
also see Chen \cite{JC}.
Motivated by the earlier results, the focus of this paper is on the full Euler equations for
the infinitely nozzle problem with general varying cross-sections by developing some useful new
techniques. Some further related results can be found in Bae-Feldman \cite{BF},
Canic-Keyfitz-Lieberman \cite{CKL},
Glimm-Ji-Li-Zhang-Zheng \cite{GJLZ}, Serre \cite{Serre}, Yuan \cite{Yuan},
and the references cited therein.

We remark that the main difference between our results and those in \cite{xx1}
is that our results allow the varying entropy function, so that the far behavior
of the density and the equation for the stream function is not only determined
by the Bernoulli function, but also by the entropy function.
Thus, it is not clear whether one can directly use the implicit function theorem
to obtain the density with respect to the Bernoulli function at the upstream,
which is the starting point of our study of this problem.
Furthermore, it is not direct to see how the maximum principle can be employed
to locate these solutions of the stream function in the physical interval
and then to extend the existence for small enough momenta which is obtained by
the standard energy estimates to the critical mass flux only by simply making
the assumption on the Bernoulli function on the boundary as in \cite{xx1}.
In this paper, for the far field behavior, we introduce the ratio of the
entropy and Bernoulli function; then, by carefully defining the upper and
lower bounds for the pressure, we find the far field behavior of the pressure
with respect to the others. In order to use the maximum principle, we extend the entropy
function via a special form, under which we find a condition on the ratio of
the entropy and Bernoulli function with some power which looks like but not
exactly the condition on the change of the momenta on the boundary at the upstream.
With the uniform estimates from the maximum principle,
we extend the existence of solutions to the critical mass flux.

The organization of the paper is as follows.
In Section 2, we reformulate the problem as {\it Problem 2} by deriving the governing equation
and boundary conditions for full Euler flows in terms of the stream function,
provided that the Euler flow has a simple topological structure and satisfies
the asymptotic behavior \eqref{1.18}--\eqref{1.21}.
In Section 3, the existence of solutions to a modified elliptic problem is established.
Subsequently, in Section 4, we analyze the asymptotic behavior of solutions in a larger class
and show the uniqueness of the solution to the boundary value problem.
This yields the existence of solutions of the boundary value problem for the stream functions.
In Section 5,  the existence and uniqueness of solutions of Problem 2 are established,
and, in Section 6, some refined estimates for the stream function is derived.
The proof of the main theorem (Theorem 1.1) except the existence part of the critical mass flux
is provided based on the results of Sections 2--5.
In Section 8, we obtain the critical mass flux.
Combining these estimates with the asymptotic behavior obtained in Section 4
yield the existence of full Euler flows which satisfy all the properties
in Theorem 1.1.

\section{Reformulation of the Problem for Stream Functions}
In this section we introduce the stream functions for the two-dimensional steady compressible
full Euler flows and derive an equivalent formulation for the full Euler flows in the nozzles.

\subsection{Equations}
It follows from \eqref{1.1} that there exists a stream function $\psi$ such that
\be\label{2.1}
\psi_{x_{1}}=-\rho v,\quad \psi_{x_{2}}=\rho u.
\ee
Furthermore, from \eqref{1.9}, we have
\be\label{2.2}
p=\frac{\gamma-1}{\gamma}\mathcal{S}(\psi)\rho^{\gamma}.
\ee
By \eqref{1.14a}, the Bernoulli law can be also written as
\be\label{2.3}
\frac{1}{2}|\nabla\psi|^{2}+\mathcal{S}(\psi)\rho^{\gamma+1}=\mathcal{B}(\psi)\rho^{2}.
\ee
In the subsonic region, we have
$$
|\nabla\psi|^{2}<c^{2}\rho^{2}=(\gamma-1)\mathcal{S}(\psi)\rho^{\gamma+1},
$$
which implies
$$
\rho^{\gamma-1}>\frac{2\mathcal{B}}{(\gamma+1)\mathcal{S}}.
$$
Let $\chi=\frac{1}{2}|\nabla\psi|^{2}$ and
$g(\rho,\psi)=\mathcal{B}(\psi)\rho^{2}-\mathcal{S}(\psi)\rho^{\gamma+1}$.
We obtain
$$
\frac{\partial g}{\partial\rho}
=\Big(\frac{2\mathcal{B}}{\mathcal{S}}-(\gamma+1)\rho^{\gamma-1}\Big)\rho \mathcal{S}<0
$$
in the subsonic region. Hence, by the implicit function theorem, there exists a unique
$\rho=\rho(\chi,\psi)$ such that
\begin{equation}\label{2.4a}
\chi=g(\rho,\psi)=\mathcal{B}(\psi)\rho^{2}-\mathcal{S}(\psi)\rho^{\gamma+1}.
\end{equation}
From \eqref{2.4a}, we have
$$
\rho_{\chi}=-\frac{1}{(\gamma+1)\mathcal{S}\rho^{\gamma}-2\mathcal{B}\rho},
\qquad
\rho_{\psi}=\frac{\mathcal{B}'\rho-\mathcal{S}'\rho^{\gamma}}{(\gamma+1)\mathcal{S}\rho^{\gamma-1}-2\mathcal{B}}.
$$
Then we have
\begin{eqnarray}
\rho_{x_{1}}&=&\rho_{\chi}(\psi_{x_{1}}\psi_{x_{1}x_{1}}+\psi_{x_{2}}\psi_{x_{2}x_{1}})+\rho_{\psi}\psi_{x_{1}}\nonumber \\[2mm]
&=&\frac{-\psi_{x_{1}}\psi_{x_{1}x_{1}}-\psi_{x_{2}}\psi_{x_{2}x_{1}}+\psi_{x_{1}}
(\mathcal{B}'\rho^{2}-\mathcal{S}'\rho^{\gamma+1})}{(\gamma+1)\mathcal{S}\rho^{\gamma}-2\mathcal{B}\rho}, \label{2.5}
\end{eqnarray}
\begin{eqnarray}
\rho_{x_{2}}&=&\rho_{\chi}(\psi_{x_{1}}\psi_{x_{1}x_{2}}+\psi_{x_{2}}\psi_{x_{2}x_{2}})
+\rho_{\psi}\psi_{x_{2}}\nonumber\\[2mm]
&=&\frac{-\psi_{x_{1}}\psi_{x_{1}x_{2}}-\psi_{x_{2}}\psi_{x_{2}x_{2}}+\psi_{x_{2}}
(\mathcal{B}'\rho^{2}-\mathcal{S}'\rho^{\gamma+1})}{(\gamma+1)\mathcal{S}\rho^{\gamma}-2\mathcal{B}\rho}. \label{2.6}
\end{eqnarray}

Now we reduce the Euler system into a second-order nonlinear
equation. Multiplying  equation \eqref{1.3a} by
$(\gamma+1)\mathcal{S}\rho^{\gamma}-2\mathcal{B}\rho$, using
expressions \eqref{2.5}--\eqref{2.6}, and making algebraic
manipulations, we obtain \be
\psi_{x_2}\big(a_{ij}(\psi,\nabla\psi)\psi_{x_ix_j}-F(\psi,\nabla\psi)\big)=0,
\ee where
\begin{eqnarray*}
&&a_{11}(\psi,\nabla\psi)=(\gamma-1)\mathcal{S}\rho^{\gamma+1}-\psi^2_{x_2},\\[1.5mm]
&&a_{12}(\psi,\nabla\psi)=a_{21}(\psi,\nabla\psi)=\psi_{x_1}\psi_{x_2},\\[1.5mm]
&&a_{22}(\psi,\nabla\psi)=(\gamma-1)\mathcal{S}\rho^{\gamma+1}-\psi^2_{x_1},\\[1.5mm]
&&F(\psi,\nabla\psi)=\frac{\gamma-1}{\gamma}\rho^{\gamma+3}\big(\gamma\mathcal{S}\mathcal{B}'-
2\mathcal{S}'\mathcal{B}+\mathcal{S}\mathcal{S}'\rho^{\gamma-1}\big).
\end{eqnarray*}
If $u>0$ in $\Omega$, then $\psi_{x_2}>0$. Thus we have
\be\label{2.8}
a_{ij}(\psi,\nabla\psi)\psi_{x_ix_j}=F(\psi,\nabla\psi). \ee
Multiplying \eqref{2.8} by
$\big((\gamma+1)\mathcal{S}\rho^{\gamma+3}-2\mathcal{B}\rho^{4}\big)^{-1}$,
we obtain \be \label{2.9}
\nabla\cdot(\frac{\nabla\psi}{\rho})=\mathcal{B}'\rho-\frac{1}{\gamma}\mathcal{S}'\rho^{\gamma}.
\ee

In summary, we have the following proposition.
\begin{proposition}\label{prop:2.1}
For any smooth flow away from the vacuum state in the nozzle $\Omega$ satisfying \eqref{1.4}--\eqref{1.6},
if the flow is globally subsonic and
\be \label{2.10}
u>0 \qquad \text{in}\ \Omega.
\ee
Then the new system formed by \eqref{2.1}--\eqref{2.3} and \eqref{2.9} is equivalent
to the original Euler equations \eqref{1.1}--\eqref{1.4a}.
\end{proposition}

The previous derivation is obviously invertible for the subsonic flow,
so we omit the details of the proof for Proposition \ref{prop:2.1}.
In order to establish the existence of solutions to system \eqref{1.1}--\eqref{1.4a},
it suffices to establish the existence of solutions to system \eqref{2.1}--\eqref{2.3} and \eqref{2.9}
satisfying \eqref{2.10}.

\medskip
\subsection{Relations between $\mathcal{S}(\psi)$, $\mathcal{B}(\psi)$,
and the Asymptotic Behavior of $\psi$ at $x_1\to\pm\infty$}

First, it follows from \eqref{1.7} that the nozzle walls are streamlines,
so $\psi$ is constant on each wall.
By \eqref{1.8} and the fact that $\psi_{x_2}>0$ since $u>0$, we have
\be
0<\psi<m\,\,\, \text{in}\ \Omega,\qquad \psi=0\,\,\, \text{on}\  W_1,\qquad \psi=m\,\,\, \text{on}\  W_2.
\ee

\smallskip
Then we study the density-speed relation by using the entropy
relation \eqref{2.2} and the Bernoulli law \eqref{2.3}. Here, unlike
isentropic flow which does not need to study the entropy relation,
i.e., the entropy function $S(x_2)$, we start from the ratio of
these two functions $B$ and $S$ as follows.

Let $D(x_2)=(BS^{-\frac{1}{\gamma}})(x_2)
$. For any $s>0$,
$\bar{\mathfrak{p}}(s)=\frac{\gamma-1}{\gamma}s^{\frac{\gamma}{\gamma-1}}>0$
is the unique solution of
$$
\big(\frac{\gamma\bar{\mathfrak{p}}(s)}{\gamma-1}\big)^{\frac{\gamma-1}{\gamma}}=s.
$$
Moreover, from \eqref{2.3}, the speed
$$
q(p,x_2;s)=\sqrt{2S^{\frac{1}{\gamma}}(x_2)\Big(s-\big(\frac{\gamma
p}{\gamma-1}\big)^{\frac{\gamma-1}{\gamma}}\Big)}.
$$
Hence, for fixed $s$ and $x_2\in[0,1]$, $q$ is a strictly decreasing
function of $p$ on $[0,\bar{\mathfrak{p}}(s)]$. By the definition of
$\bar{\mathfrak{p}}(s)$, one has
$$
q(\bar{\mathfrak{p}}(s),x_2;s)=0<c(\bar{\mathfrak{p}}(s),x_2).
$$
Now we claim that $q(0,x_2;s)>c(0,x_2)$. Indeed,
$$
q(p,x_2;s)\rightarrow\sqrt{2S^{\frac{1}{\gamma}}(x_2)s}>0  \qquad
\mbox{as $p\rightarrow0$},
$$
and, by the definition of sonic speed, $c(0,x_2)=0$.
Thus, $q(0,x_2)>0=c(0,x_2)$.
This completes the claim.

Since
$c^{2}(p,x_2)=(\gamma-1)S(x_{2})\rho^{\gamma-1}=\big(\gamma^{\gamma-1}(\gamma-1)S(x_2)\big)^{\frac{1}{\gamma}}
p^{\frac{\gamma-1}{\gamma}}$ is an increasing function of $p$, there
exists a unique $\mathfrak{p}(s)\in[0,\bar{\mathfrak{p}}(s)]$ such
that
$$
c^{2}(\mathfrak{p}(s),x_2)=q^{2}(\mathfrak{p}(s),x_2;s).
$$
More precisely,
$$
\mathfrak{p}(s)=\frac{\gamma-1}{\gamma}\big(\frac{2s}{\gamma+1}\big)^{\frac{\gamma}{\gamma-1}}.
$$

In summary, we have

\begin{lemma}\label{lem:2.1}
There exist
$\bar{\mathfrak{p}}=\bar{\mathfrak{p}}(s)$,
$\mathfrak{p}=\mathfrak{p}(s)$, and $\Gamma=\Gamma(s,x_2)$ such that
\begin{eqnarray}
&& S^{\frac{1}{\gamma}}(x_2)\big(\frac{\gamma\bar{\mathfrak{p}}(D(x_2))}{\gamma-1}\big)^{\frac{\gamma-1}{\gamma}}=B(x_2), \label{2.12a}\\
&& S^{\frac{1}{\gamma}}(x_2)\big(\frac{\gamma\mathfrak{p}(s)}{\gamma-1}\big)^{\frac{\gamma-1}{\gamma}}+\frac{\Gamma^{2}(s,x_2)}{2}=B(x_2), \label{2.12b} \\
&& c^{2}(\mathfrak{p}(s),x_2)=\Gamma^{2}(s,x_2), \label{2.12c}
\end{eqnarray}
where $\bar{\mathfrak{p}}(s)$, $\mathfrak{p}(s)$, and
$\Gamma(s,x_2)$ are the maximum pressure, critical pressure, and
critical speed, respectively, for the fixed ratio $s$ of the
Bernoulli function and the entropy function.
\end{lemma}

Then direct calculations show that
$$
\frac{d\bar{\mathfrak{p}}}{ds}>0,\qquad\ \frac{d\mathfrak{p}}{ds}>0.
$$
Clearly, $\mathfrak{p}(s)<\bar{\mathfrak{p}}(s)$ for $s>0$. By the
continuity and monotonicity of $\mathfrak{p}(s)$ and
$\bar{\mathfrak{p}}(s)$, there exists a unique
$\underline{\delta}>0$ such that \be\label{2.14}
\mathfrak{p}(\underline{D}+\underline{\delta})=\bar{\mathfrak{p}}(\underline{D}).
\ee where $\underline{D}=\inf\limits_{x_2\in[0,1]}D(x_2)$.

Moreover, it follows from \eqref{2.14} that there exists a uniform constant $C>0$ such that
\bee\label{2.15}
\begin{cases}
C^{-1}\leq\mathfrak{p}(\underline{D})<\bar{\mathfrak{p}}(\underline{D})=\mathfrak{p}(\underline{D}+\underline{\delta})\leq C,\\[2mm]
C^{-1}\leq\mathfrak{p}'(s)\leq C,\ C^{-1}\leq\bar{\mathfrak{p}}'(s)\leq C\qquad\,\, \text{if}\ s\in(\underline{D},\underline{D}+\underline{\delta}),\\[2mm]
C^{-1}\leq S^{\frac{1}{\gamma}}(x_2)\left(\frac{\gamma
p}{\gamma-1}\right)^{\frac{\gamma-1}{\gamma}}\leq
C\qquad\qquad\,\,\,\, \text{if}\
p\in(\mathfrak{p}(\underline{D}),\bar{\mathfrak{p}}(\underline{D}+\underline{\delta})).
\end{cases}
\eee
Hereafter, $C$ denotes a generic constant which depends only
essentially on $S$ and $B$.

If $S(x_{2})$ and  $B(x_{2})$ satisfy \be \label{2.23}
\|(S-\underline{S},B-\underline{B})\|_{C^{1,1}([0,1])}\leq\delta,
\ee then \be
\bar{D}=\sup\limits_{x_{2}\in[0,1]}\frac{B(x_{2})}{S^{1/\gamma}(x_{2})}\leq
\underline{D}+ C\delta. \ee

Finally, we study the behavior of $\mathcal{S}$ and $\mathcal{B}$ at
the upstream and downstream in the far fields of the nozzle where
the flow may have certain simple structure. Indeed, for the flows
satisfying \eqref{1.18}--\eqref{1.24}, one can determine $p_0$,
$p_1$, $\rho_{0}(x_2)$, $\rho_{1}(x_2)$, $u_{0}(x_{2})$, and
$u_{1}(x_{2})$ first.

If the flow satisfies \eqref{1.18}, then
\be\label{2.18}
\frac{u^{2}_{0}(x_{2})}{2}+S(x_2)\rho^{\gamma-1}_0(x_2;p_0)=B(x_{2}),\quad
u_0(x_2)>0,\quad \rho_0(x_2;p_0)=\left(\frac{\gamma
p_0}{(\gamma-1)S(x_2)}\right)^{\frac{1}{\gamma}}
\ee
and
\be\label{2.19}
\int_{0}^{1}\rho_0(x_2;p_0) u_{0}(x_{2})dx_{2}=m,
\ee
which imply that
\be
u_{0}(x_{2})=\sqrt{2\big(B(x_{2})-S(x_2)\rho_{0}^{\gamma-1}(x_2;p_0)\big)},
\ee and \be\label{2.21}
m=\int_{0}^{1}\rho_{0}(x_2;p_0)\sqrt{2\big(B(x_{2})-S(x_2)\rho_{0}^{\gamma-1}(x_2;p_0)\big)}dx_{2}.
\ee

\begin{lemma}\label{lem:2.2}
Let $\delta\leq \frac{\underline{\delta}}{2}$. It follows from
\eqref{2.14} that $\mathfrak{p}(D(x_{2}))\leq
\mathfrak{p}(\bar{D})<\bar{\mathfrak{p}}(\underline{D})$. Then we
have
\begin{itemize}
\item[(i)] For given $S(x_{2})$, $B(x_{2})$, and $m>0$,
\eqref{2.21} has a solution $p_0$ satisfying
$p_{0}\in(\mathfrak{p}(\bar{D}),\bar{\mathfrak{p}}(\underline{D}))$
such that $\rho_0(x_2;p_0)$ and $u_0(x_2)>0$ satisfy \eqref{2.18}
for $x_2\in[0,1]$;

\item[(ii)] If $\underline{S}>0$ and
$\|B-\underline{B}\|_{C^{1,1}([0,1])}=\delta\leq \hat{\delta}_{0}$
for some small $\hat{\delta}_0$, then there is a positive constant $C$ such that
\bee \label{2.21a}
\begin{cases}
C^{-1}\delta^{2\beta}\leq\bar{\mathfrak{p}}(\underline{D})-p_{0}\leq C,\\[2mm]
C^{-1}\delta^{\beta}\leq u_{0}(x_2)\leq C,\\[2mm]
|u_{0}'(x_{2})|\le \frac{|B'(x_{2})|+\frac{1}{\gamma}|S'(x_{2})|\rho^{\gamma-1}_{0}(x_2;p_0)}{u_0(x_2)}\leq
C\delta^{1-\beta}.
\end{cases}
\eee
\end{itemize}
\end{lemma}

The proof of this lemma is as follows.

Result (i) is for obtaining a global subsonic flow in the nozzle.
Clearly, from \eqref{2.18}
$$
\frac{d}{dp_{0}}\Big(\int_{0}^{1}\rho_{0}(x_2;p_0))\sqrt{2\big(B(x_{2})-S(x_2)\rho_{0}^{\gamma-1}(x_2;
p_0)\big)}dx_{2}\Big)<0, \quad \frac{d\rho_0(x_2;p_0)}{dp_0}>0
$$
for $p_{0}\in(\mathfrak{p}(\bar{D}),\bar{\mathfrak{p}}(\underline{D}))$.
Let
$$
\varrho(D;x_2):=\left(\frac{\gamma
\mathfrak{p}(D)}{(\gamma-1)S(x_2)}\right)^{\frac{1}{\gamma}}
\quad\text{and}\quad
\bar{\varrho}(D;x_2):=\left(\frac{\gamma
\bar{\mathfrak{p}}(D)}{(\gamma-1)S(x_2)}\right)^{\frac{1}{\gamma}}.
$$
It follows from \eqref{2.15} and \eqref{2.23} that
\begin{eqnarray*}
&&\int_{0}^{1}\bar{\varrho}(\underline{D};x_2)\sqrt{2S^{1/\gamma}(x_2)\big(D(x_{2})
  -(S^{1/\gamma}\bar{\varrho}(\underline{D};x_2))^{\gamma-1}\big)}dx_{2}\\
&&=\int_{0}^{1}\bar{\varrho}(\underline{D};x_2)\sqrt{2S^{1/\gamma}(x_2)\big(D(x_{2})-\underline{D}\big)}dx_{2}\\
&&\leq C\delta^{1/2}.
\end{eqnarray*}
In addition,
\begin{eqnarray*}
&&\int_{0}^{1}\varrho(\overline{D};x_2)\sqrt{2S^{1/\gamma}(x_2)\big(D(x_{2})-(S^{1/\gamma}(x_2)\varrho(\overline{D};x_2))^{\gamma-1}\big)}dx_{2}\\
&&\geq\int_{0}^{1}\varrho(\overline{D};x_2)\sqrt{2S^{1/\gamma}(x_2)\big(\underline{D}-(S^{1/\gamma}\varrho(\overline{D};x_2))^{\gamma-1}\big)}dx_{2}\\
&&=\int_{0}^{1}\varrho(\overline{D};x_2)
  \sqrt{2S(x_2)\big(\bar{\varrho}(\underline{D};x_2)^{\gamma-1}-\varrho(\overline{D};x_2)^{\gamma-1}\big)}dx_{2}\\
&&=\int_{0}^{1}\varrho(\overline{D};x_2)\sqrt{2S(x_2)
 \big(\varrho(\underline{D}+\underline{\delta};x_2)^{\gamma-1}-\varrho(\overline{D};x_2)^{\gamma-1}\big)}dx_{2}\\
&&\geq \int_{0}^{1}\varrho(\overline{D};x_2)\sqrt{2S(x_2)\big(\varrho(\underline{D}+\underline{\delta};x_2)^{\gamma-1}
-\varrho(\underline{D}+\underline{\delta}/2;x_2)^{\gamma-1}\big)}dx_{2}\\
&&\geq C^{-1}\underline{\delta}^{1/2}.
\end{eqnarray*}
Therefore, for any $\beta\in(0,1/3)$, there exists
$\tilde{\delta}_{0}\in(0,\underline{\delta}/2)$ such that
\eqref{2.21} has a unique solution
$p_{0}\in(\mathfrak{p}(\bar{D}),\bar{\mathfrak{p}}(\underline{D}))$,
if $0\leq\delta\leq\tilde{\delta}_{0}$ and
$m\in(\delta^{\beta},m_{1})$, where $m_{1}$ satisfies
$C^{-1}\underline{\delta}^{1/2}\geq m_{1} \geq
2\tilde{\delta}^{\beta}_{0}>C\delta^{1/2}$. Later on, for
simplicity, we will choose $\beta=1/4$. However, all the results
hold for $\beta\in(0,1/3)$.

By virtue of \eqref{2.21}, one has
\begin{equation*}
\begin{array}{lll}
m&=&\int_{0}^{1}\rho_{0}(x_2;p_0)
\sqrt{2S^{1/\gamma}(x_2)\big(D(x_{2})-(S^{1/\gamma}(x_2)\rho_{0}(x_2;p_0))^{\gamma-1}\big)}dx_{2}\\[2mm]
&=&\int_{0}^{1}\rho_{0}(x_2;p_0)\sqrt{2S^{1/\gamma}(x_2)\big(D(x_{2})-\underline{D}+\underline{D}
-(S^{1/\gamma}(x_2)\rho_{0}(x_2;p_0))^{\gamma-1}\big)}dx_{2}\\[2mm]
&\leq&C\int_{0}^{1}\rho_{0}(x_2;p_0)\sqrt{\delta+(S^{1/\gamma}(x_2)\bar{\varrho}(\underline{D};x_2))^{\gamma-1}
-(S^{1/\gamma}(x_2)\rho_{0}(x_2;p_0))^{\gamma-1}}dx_{2}.
\end{array}
\end{equation*}
Thus, we have
$$
\delta+(S^{1/\gamma}(x_2)\bar{\varrho}(\underline{D};x_2))^{\gamma-1}-(S^{1/\gamma}(x_2)\rho_{0}(x_2;p_0))^{\gamma-1}\geq
C^{-1}\delta^{2\beta}.
$$
Since $\beta<1/3$, then there exists $\hat{\delta}_{0}\in(0,\tilde{\delta}_{0})$
such that, if $0<\delta\leq\hat{\delta}_{0}$, then
$$
\bar{\varrho}(\underline{D};x_2)^{\gamma-1}-\rho_{0}^{\gamma-1}(x_2;p_0)\geq C^{-1}\delta^{2\beta}.
$$
Consequently, if $\underline{S}>0$ and $\|B-\underline{B}\|_{C^{1,1}([0,1])}=\delta\leq\hat{\delta}_{0}$,
there exists a constant $C>0$ such that \eqref{2.21a} holds.
This completes the proof of Lemma \ref{lem:2.2}.

\medskip
Next, to determine the asymptotic states in the downstream, we parameterize the streamlines in the downstream
by their positions in the upstream. Using \eqref{1.18}, \eqref{1.20}, and \eqref{2.10},
we define
\be\label{2.25}
y=y(s)\qquad \text{for}\ s\in[0,1],
\ee
such that
\bee
&S(s)\rho_{0}^{\gamma-1}(s;p_0)+\frac{u^{2}_{0}(s)}{2}
 =S(s)\rho_{1}^{\gamma-1}(s;p_0)+\frac{u^{2}_{1}(y(s))}{2},\quad u_{1}(y(s))>0,\label{2.26}\\[2mm]
&\int_{0}^{s}\rho_{0}(t;p_0)u_{0}(t)dt=\int_{0}^{y(s)}\rho_{1}(t;p_1)u_{1}(t)dt,\quad
\rho_i(s;p_i):=\left(\frac{\gamma
p_i}{(\gamma-1)S(s)}\right)^{\frac{1}{\gamma}},\label{2.27}\\[2mm]
&y(0)=a,\quad y(1)=b.\label{2.28} \eee Then the streamline which
starts at $(-\infty,s)$ ends at $(\infty,y(s))$.

The next procedure is similar as before, where we consider $\rho_i$
instead of $p_i$ for the monotone relationship between them and for
simplicity by recalling
$$
\varrho(D; x_2):=\left(\frac{\gamma
\mathfrak{p}(D)}{(\gamma-1)S(x_2)}\right)^{\frac{1}{\gamma}}
\quad\text{and}\quad\bar{\varrho}(D;x_2):=\left(\frac{\gamma
\bar{\mathfrak{p}}(D)}{(\gamma-1)S(x_2)}\right)^{\frac{1}{\gamma}}.
$$

The mapping in \eqref{2.25} is well-defined due to condition
\eqref{2.26} and \eqref{2.27}. In fact, \eqref{2.27} deduces that
\be\label{2.31b}
\rho_{0}(s;p_0)u_{0}(s)= \rho_{1}(s;p_1)u_{1}(y(s))y'(s).
\ee
If
$\rho_1<\bar{\varrho}(\underline{D};x_2)\leq\bar{\varrho}(D;x_2)$, then
$$
S(s)(\rho_{0}^{\gamma-1}(s;p_0)-\rho_{1}^{\gamma-1}(s;p_1))+\frac{u^2_0(s)}{2}
=B(s)-S(s)\rho_{1}^{\gamma-1}(s)>
B(s)-S^{1/\gamma}(s)D(s)=0.
$$
Then we have
\bee
\begin{cases}
\frac{dy}{ds}=\frac{\rho_{0}(s;p_0)u_{0}(s)}{\rho_{1}(s;p_1)\sqrt{2S(s)(\rho_{0}^{\gamma-1}(s;p_0)-\rho_{1}^{\gamma-1}(s;p_1))
+ u^{2}_{0}(s)}},\\
y(0)=a,
\end{cases}
\eee
where the pressure in the downstream $p_{1}$ satisfies
\be\label{2.31}
\int_{0}^{1}\frac{\rho_{0}(s;p_0)u_{0}(s)}{\rho_{1}(s;p_1)\sqrt{2S(s)(\rho_{0}^{\gamma-1}(s;p_0)-\rho_{1}^{\gamma-1}(s;p_1))
+ u^{2}_{0}(s)}}ds=b-a.
\ee
It remains to show that there exists
$p_{1}\in(\mathfrak{p}(\bar{D}),\bar{\mathfrak{p}}(\underline{D}))$
satisfying \eqref{2.31}. As in the proof to Lemma \ref{lem:2.2}, we
find that, for $p_{1}\in(\mathfrak{p}(\bar{D}),\bar{\mathfrak{p}}(\underline{D}))$,
$$
\frac{d}{dp_{1}}\int_{0}^{1}\frac{\rho_{0}(s;p_0)u_{0}(s)}{\rho_{1}(s;p_1)
\sqrt{2S(s)(\rho_{0}^{\gamma-1}(s;p_0)-\rho_{1}^{\gamma-1}(s;p_1))
+u^{2}_{0}(s)}}ds>0.
$$
On one hand, there exists $\bar{\delta}_{0}\in(0,\tilde{\delta}_{0})$ such that, if $\delta\leq\bar{\delta}_0$, then
\begin{equation*}
\begin{array}{lll}
&&\int_{0}^{1}\frac{\rho_{0}(s;p_0)u_{0}(s)}{\bar{\varrho}(\underline{D};s)
\sqrt{2S(s)(\rho_{0}^{\gamma-1}(s;p_0)-\bar{\varrho}(\underline{D};s)^{\gamma-1})
+u^{2}_{0}(s)}}ds\\[2mm]
&&=\int_{0}^{1}\frac{\rho_{0}(s;p_0)u_{0}(s)}{\bar{\varrho}(\underline{D};s)\sqrt{2S^{1/\gamma}(s)(D(s)-\underline{D})}}ds\\[2mm]
&&\geq C\delta^{(2\beta-1)/2}>b-a.
\end{array}
\end{equation*}
On the other hand,
\begin{equation}
\begin{array}{lll}
&&\int_{0}^{1}\frac{\rho_{0}(s;p_0)u_{0}(s)}{\varrho(\overline{D};s)
\sqrt{2S(s)\big(\rho_{0}^{\gamma-1}(s;p_0)-\varrho(\overline{D};s)^{\gamma-1}\big)
+u^{2}_{0}(s)}}ds\\[2mm]
&&=\int_{0}^{1}\frac{\rho_{0}(s;p_0)u_{0}(s)}{\varrho(\overline{D};s)\sqrt{2S^{1/\gamma}(s)(D(s)-(S^{1/\gamma}\rho(\overline{D}))^{\gamma-1})}}ds\\[2mm]
&&\leq\left(\frac{\gamma-1}{\gamma\mathfrak{p}(\overline{D})}\right)^{\frac{1}{\gamma}}
  \int_{0}^{1}\frac{S^{\frac{1}{2\gamma}}(s)\rho_{0}(s;p_0)u_{0}(s)}{\sqrt{2(\underline{D}-(S^{1/\gamma}\rho(\overline{D}))^{\gamma-1})}}ds\\[2mm]
&&=\left(\frac{\gamma-1}{\gamma\mathfrak{p}(\overline{D})}\right)^{\frac{1}{\gamma}}
  \int_{0}^{1}\frac{S^{\frac{1}{2\gamma}}(s)\rho_{0}(s;p_0)u_{0}(s)}
  {\sqrt{2\big((\frac{\gamma\bar{\mathfrak{p}}(\underline{D})}{\gamma-1})^{\frac{\gamma-1}{\gamma}}
  -(\frac{\gamma\mathfrak{p}(D)}{\gamma-1})^{\frac{\gamma-1}{\gamma}}\big)}}ds\\[2mm]
&&\leq\left(\frac{\gamma-1}{\gamma\mathfrak{p}(\overline{D})}\right)^{\frac{1}{\gamma}}
  \frac{m}{\sqrt{2\big((\frac{\gamma\mathfrak{p}(\underline{D}+\underline{\delta})}{\gamma-1})^{\frac{\gamma-1}{\gamma}}
  -(\frac{\gamma\mathfrak{p}(D)}{\gamma-1})^{\frac{\gamma-1}{\gamma}}\big)}}
  \max_{x_2\in[0,1]}\big(S^{\frac{1}{2\gamma}}(x_2)\big)
  \\[2mm]
&&=\left(\frac{\gamma-1}{\gamma\mathfrak{p}(\overline{D})}\right)^{\frac{1}{\gamma}}
  \frac{m}{\sqrt{\frac{4}{\gamma+1}(\underline{D}+\underline{\delta}-D)}}\max_{x_2\in[0,1]}\big(S^{\frac{1}{2\gamma}}(x_2)\big)\\[2mm]
&&\leq \frac{C m}{\underline{\delta}^{1/2}}<b-a.
 \end{array}
\end{equation}
Thus, there exists a unique
$p_{1}\in(\mathfrak{p}(\bar{D}),\bar{\mathfrak{p}}(\underline{D}))$
such that \eqref{2.31} holds, provided that
$0\leq\delta\leq\bar{\delta}_{0}$ and $m\in(\delta^{\beta},m_{2})$
for some $\bar{\delta}_0$ small enough and
$2\bar{\delta}_0^{\beta}\leq
m_{2}\leq\min\{m_{1},C^{-1}(b-a)\underline{\delta}^{1/2}\}$. Once
$p_{1}$ determined, $y(s)$, $\rho_1(s;p_1)$, and $u_{1}(s)$ can be obtained
from \eqref{2.26}, \eqref{2.27}, and \eqref{2.31b}. Therefore, the
above calculations yield the following proposition.

\begin{proposition}\label{2}
Let $\underline{S},\ \underline{B}>0$, and $D$ be the ratio of
$S^{1/\gamma}$ and $B$. There exists $\bar{\delta}_{0}>0$ such that,
for any $S, B\in C^{1,1}([0,1])$ satisfying \eqref{2.23} with
$\delta\leq \bar{\delta}_{0}$ respectively, there exists
$\bar{m}\geq2\bar{\delta}^{\beta}_{0}$, $\beta\in(0,\frac{1}{3})$,
such that

\begin{itemize}
\item[(i)] Existence: There exists solutions $(\rho_{0},u_{0}, p_0)$ to \eqref{2.18}--\eqref{2.19} and $(\rho_{1},u_{1}, p_1)$
to \eqref{2.26}--\eqref{2.28} if $m\in(\delta^{\beta},\bar{m})$ with $\rho_j^\gamma(y_j;p_j)=\frac{\gamma}{\gamma-1}\frac{p_j}{S(y_j)}, j=0,1$,
$y_0=x_2$ and $y_1=y(s)$;

\item[(ii)] Subsonicity: $p_{0},\ p_{1}\in(\mathfrak{p}(\bar{D}),\bar{\mathfrak{p}}(\underline{D}))$;

\item[(iii)] Limiting behaviors: Either $p_{0}\rightarrow \mathfrak{p}(\bar{D})$ or $p_{1}\rightarrow\mathfrak{p}(\bar{D})$ as $m\rightarrow\bar{m}$,
\end{itemize}
where $\bar{D}=\sup\limits_{x_{2}\in[0,1]}D(x_{2})$ and
\begin{equation}\label{2.31a}
\bar{m}=\sup\{s\, :\, m\in(\delta^{\beta},s)\ \text{such that there
exist}\ p_0,\
p_1\in(\mathfrak{p}(\bar{D}),\bar{\mathfrak{p}}(\underline{D}))\}.
\end{equation}
\end{proposition}

\begin{proof}  Results (i)--(ii) are direct corollaries of Lemmas \ref{lem:2.1}--\ref{lem:2.2}.
It suffices to verify (iii).

For $m\in(\delta^{\beta},m_2)$,
$p_0,p_1\in(\mathfrak{p}(\bar{D}),\bar{\mathfrak{p}}(\underline{D}))$.
For fixed $S(x_2)$ and $B(x_2)$, $p_0$ decreases as $m$ increases.
If
$$
m\rightarrow\tilde{m}
=\int_0^1\varrho(\overline{D};x_2)\sqrt{2S^{1/\gamma}(x_2)\big(D(x_2)-(S^{1/\gamma}\varrho(\overline{D}))^{\gamma-1}\big)}
dx_2,
$$
then
$$
p_0\rightarrow\mathfrak{p}(\overline{D}).
$$
For $\bar{m}$ defined in \eqref{2.31a}, $\bar{m}\in[m_2,\tilde{m}]$.
Note that both $p_0$ and $p_1$ are uniformly away from
$\bar{\mathfrak{p}}(\underline{D})$. If neither $p_0$ nor $p_1$
approaches to $\mathfrak{p}(\overline{D})$ as $m\rightarrow
\bar{m}$, then there always exist
$p_0,p_1\in(\mathfrak{p}(\bar{D}),\bar{\mathfrak{p}}(\underline{D}))$
for $m\in (\delta^{\beta},\bar{m}+\epsilon)$ for some small positive
$\epsilon$, which contradicts with the definition of $\bar{m}$. This
completes the proof.
\end{proof}

\subsection{Reformulation of Problem 1: Problem 2}

Let $X_2$ be the coordinate in the upstream. Since
$\rho_{0}(X_2;p_0)u_{0}(X_2)>0$ for $X_2\in[0,1]$, $\psi$ is an
increasing function of $X_{2}$. Thus, we can represent $X_{2}$ as a
function of $\psi$, which is defined by
$$
X_{2}=\kappa(\psi), \qquad 0\leq\psi\leq m,
$$
so that
$$
\psi(X_2)=\int_{0}^{X_{2}}\rho_{0}(s;p_0)u_{0}(s)ds.
$$

It follows from  Proposition 2.1 that, if \eqref{2.10} holds in $\Omega$,
through each point $(x_{1},x_{2})\in\Omega$, there exists a unique
streamline which starts from the upstream.
Along each streamline, the stream function is a constant by the definition.
Therefore, through any $(x_{1},x_{2})$ in the nozzle, there exists a unique streamline
from $(-\infty,\kappa(\psi))$ with $\psi=\psi(x_{1},x_{2})$. Thus, we denote
$$
\mathcal{S}=S(\kappa(\psi)), \quad \mathcal{B}=B(\kappa(\psi))\,\, \qquad \mbox{for $0\leq \psi\leq m$}.
$$

Then our main task in the rest of the paper is to solve the following
problem:

\bigskip
{\bf Problem 2} (Reformulation of {\bf Problem 1}). {\it Seek a solution of the boundary value problem:
\bee\label{2.33}
\begin{cases}
 \nabla\cdot(\frac{\nabla\psi}{\rho(|\nabla\psi|^{2},\psi)})
=\mathcal{B}'(\psi)\rho(|\nabla\psi|^{2},\psi)
-\frac{1}{\gamma}\mathcal{S}'(\psi)\rho^{\gamma}(|\nabla\psi|^{2},\psi)\ \qquad \text{in}\ \Omega,\\
\psi=\frac{x_{2}-f_{1}(x_{1})}{f_{2}(x_{1})-f_{1}(x_{1})}m\qquad \ \text{on}\ \partial\Omega,
\end{cases}
\eee
such that
\begin{itemize}
\item[(i)] The flow field induced by
$$
\rho=\rho(|\nabla\psi|^{2},\psi),\quad u=\frac{\psi_{x_{2}}}{\rho},\quad  v=-\frac{\psi_{x_{1}}}{\rho},
\quad p=\frac{\gamma-1}{\gamma}\mathcal{S}(\psi)\rho^{\gamma}
$$
satisfies \eqref{1.18}--\eqref{1.24};

\item[(ii)] In the upstream, we have
\be \psi(X_2)=\int_{0}^{X_{2}}\rho_{0}(s;p_0)u_{0}(s)ds, \qquad
 0\leq\psi\leq m.
\ee
\end{itemize}
}

\section{Existence of Solutions of a Modified Boundary Value Problem}\label{sec:3}

There are three main difficulties to solve the boundary value problem \eqref{2.33}.
The first is that equation \eqref{2.33} may degenerate
at the sonic states. The second is that, although the entropy and Bernoulli function are well-defined on $[0,m]$,
the density $\rho$ is not well-defined for arbitrary $\psi$ and $|\nabla\psi|$.
The last is that the problem is in an unbounded domain.
Our basic strategy is to extend the definition of $\mathcal{S}(\psi)$ and $\mathcal{B}(\psi)$ appropriately,
introduce the elliptic cut--off to truncate $|\nabla\psi|$ in $\rho(|\nabla\psi|^{2},\psi)$ in a suitable way,
and use a sequence of bounded domains and solve the problems on it to approximate the original one.

In this section we first introduce a modified problem and then solve it,
which can be indeed used to solve the original problem with the asymptotic behavior in \S 4.

Set
\bee
a(s)=
\begin{cases}
\mathcal{S}'(s) \quad&\text{if}\ 0\leq s\leq m,\\[1.5mm]
\mathcal{S}'(m)\frac{2m-s}{m} \quad&\text{if}\ m\leq s\leq 2m,\\[1.5mm]
\mathcal{S}'(0)\frac{s+m}{m} \quad&\text{if}\ -m\leq s\leq 0,\\[1.5mm]
0,\qquad&\text{if}\ s\geq2m \ \text{or}\ s\leq-m,
\end{cases}  \nonumber
\eee
and
\bee
b(s)=
\begin{cases}
(\frac{\mathcal{B}}{\mathcal{S}^{\gamma}})'(s) \quad&\text{if}\ 0\leq s\leq m,\\[1.5mm]
(\frac{\mathcal{B}}{\mathcal{S}^{\gamma}})'(m)\frac{2m-s}{m} \quad&\text{if}\ m\leq s\leq 2m,\\[1.5mm]
(\frac{\mathcal{B}}{\mathcal{S}^{\gamma}})'(0)\frac{s+m}{m} \quad&\text{if}\ -m\leq s\leq 0,\\[1.5mm]
0,\qquad&\text{if}\ \psi\geq2m,\ \text{or}\ s\leq-m.
\end{cases}\nonumber
\eee
We define
\begin{equation}\label{3.0}
\tilde{\mathcal{S}}(s)=\mathcal{S}(0)+\int_{0}^{s}a(t)dt,\quad
\tilde{\mathcal{B}}(s)=\tilde{\mathcal{S}}^{\gamma}(s)
\Big(\frac{\mathcal{B}(0)}{\mathcal{S}^{\gamma}(0)}+\int_{0}^{s}b(t)dt\Big).
\end{equation}
Then $(\tilde{\mathcal{S}}, \tilde{\mathcal{B}})\in C^{1,1}(\mathbb{R})$.
We remark here that the definition of $b(s)$ in this particular form instead of $B'$ itself
is for some technical reason,
roughly speaking, due to the maximum principle.
Moreover, since $m>\delta^{\beta}$, there exists a suitably small $\bar{\delta}_1$ such that, when $\delta<\bar{\delta}_1$,
$$
0<\underline{B}-C\delta\leq\frac{1}{2}\tilde{u}^2_0(s)+ \tilde{\mathcal{S}}(s)\rho^{\gamma-1}_0
\leq \sup\limits_{x_2\in[0,1]}B(x_2)+C\delta, \qquad \tilde{u}_0(s)>0
$$
for some $C>0$,
where $\|(\tilde{S}-\tilde{S}(0), \tilde{B}-\tilde{B}(0))\|_{C^{1,1}(\mathbb{R})}\leq \delta^{1-\beta}$.

In the rest of the paper, we will always use the following notations:
$$
\rho_{1}(|\nabla\psi|^{2},\psi)=\frac{\partial\rho(|\nabla\psi|^{2},\psi)}{\partial|\nabla\psi|^{2}},\quad \rho_{2}(|\nabla\psi|^{2},\psi)=\frac{\partial \rho(|\nabla\psi|^{2},\psi)}{\partial\psi}.
$$
It is easy to see that
$$\rho_{1}(|\nabla\psi|^{2},\psi)=-\frac{1}{2\rho(c^{2}-|\nabla\psi|^{2}/\rho^{2})}$$
goes to $-\infty$ when the flow approaches the sonic state from
the subsonic states. To avoid it, we introduce the following cut--off function.
For $\epsilon>0$, let \bee\label{cut-off} \zeta_{0}(s)=
\begin{cases}
s \qquad&\text{if}\ s<-2\epsilon,\\
-\frac{3}{2}\epsilon \qquad&\text{if}\ s\geq-\epsilon
\end{cases}
\eee be a smooth increasing function  such that $|\zeta'_{0}|\leq1$.
We define \be\label{cut-off using}
\tilde{\Delta}^2(|\nabla\psi|^{2},\psi)
:=\zeta_{0}(|\nabla\psi|^{2}-(\gamma-1)\tilde{\mathcal{S}}(\psi)\tilde{\rho}^{\gamma+1})
+(\gamma-1)\tilde{\mathcal{S}}(\psi)\tilde{\rho}^{\gamma+1}, \ee
where \be \frac{1}{2}\tilde{\Delta}^2(|\nabla\psi|^2, \psi)+
\tilde{\mathcal{S}}(\psi)\tilde{\rho}^{\gamma+1}=\tilde{\mathcal{B}}(\psi)\tilde{\rho}^2.
\ee
A direct calculation shows
\be
\tilde{S}_{ij}(q,z)=\tilde{\rho}(|q|^{2},z)\delta_{ij}-2\tilde{\rho}_{1}(|q|^{2},z)\xi_{i}\xi_{j},
\ee
and
$$
\tilde{\rho}_{1}(|\nabla\psi|^{2},\psi)=\frac{\zeta'_{0}\tilde{\rho}}
{4\tilde{\mathcal{B}}\tilde{\rho}^{2}-(\gamma+1)^2\mathcal{S}\tilde{\rho}^{\gamma+1}+(\gamma^2-1)\zeta'_{0}
\tilde{S}\tilde{\rho}^{\gamma+1}}<0.
$$
Obviously, there exist two positive constants $\lambda(\epsilon)$ and $\Lambda(\epsilon)$ such that
\be
\lambda|\xi|^{2}\leq \tilde{S}_{ij}(q,z)\xi_{i}\xi_{j}\leq\Lambda |\xi|^{2}
\ee
for any $z\in\mathbb{R}$, $q\in\mathbb{R}^{2}$, and $\xi\in\mathbb{R}^{2}$, which means that
the modified equation is uniformly elliptic. Thus, instead of solving {\it Problem 2},
we first solve the following problem:

\bigskip
{\bf Problem 3} (Modified Problem). {\it Seek a solution to the boundary value problem:
\bee\label{3.10}
\begin{cases}
 \nabla\cdot(\frac{\nabla\psi}{\tilde{\rho}})
=\tilde{B}'\tilde{\rho}
-\frac{1}{\gamma}\tilde{\mathcal{S}}'\tilde{\rho}^{\gamma}\qquad \text{in}\ \Omega,\\[2mm]
\psi=\frac{x_{2}-f_{1}(x_{1})}{f_{2}(x_{1})-f_{1}(x_{1})}m\qquad\qquad\,\,\,\, \text{on}\ \partial\Omega
\end{cases}
\eee
such that $\|\psi\|_{C^{1,1}}$ has a uniform upper bound.}

\medskip
\begin{proposition}\label{1}
Let the boundary $\partial\Omega$ satisfy \eqref{1.3}--\eqref{1.6}.
Then there exists
$0<\delta_1\leq\min\{\bar{\delta}_0,\bar{\delta}_1\}$ such that, if
$\|(S-\underline{S}, B-\underline{B})\|_{C^{1,1}([0,1])}\leq\delta$
with $0<\delta\leq \delta_1$ and $m\in(\delta^{\beta},m_1)$ with
$m_1=2\delta^{\beta/2}\leq\bar{m}$, where $\bar{m}$ is defined in
Proposition {\rm 2.2}, then {\rm Problem 3} has a solution $\psi\in
C^{2,\alpha}(\overline{\Omega})$ satisfying \be\label{3.11}
|\psi|\leq C(\epsilon,\delta),\qquad
|\nabla\psi|^{2}\leq(\gamma-1)\mathcal{S}\rho^{\gamma+1}-2\epsilon
\ee
for some $\epsilon>0$, so
$|\nabla\psi|^{2}\leq\underline{\Sigma}(\epsilon) -2\epsilon$ with
$\underline{\Sigma}(\epsilon) :=(\gamma+1)(\underline{S}+\delta+\epsilon)
(\frac{2(\underline{B}+\delta-\epsilon)}{(\gamma+1)(\underline{S}
+\epsilon)})^{\frac{\gamma+1}{\gamma-1}}$.
\end{proposition}

\begin{proof} The proof of the existence part is standard via approximation
by the corresponding problems on bounded domains,
while inequality \eqref{3.11} is crucial here,
since $\underline{\Sigma}(\epsilon)$ depends not only on $B$ but
also on $S$ for the non-isentropic flows. We divide
the proof into four steps.

\medskip
1. First, we use a sequence of boundary value problems on bounded domains
to approximate {\it Problem 3}
on the unbounded domain.
Since the key point is to obtain estimate \eqref{3.11},
we focus on the following boundary value problem:
\bee\label{3.12}
\begin{cases}
\nabla\cdot(\frac{\nabla\psi}{\tilde{\rho}}) =\tilde{B}'\tilde{\rho}
-\frac{1}{\gamma}\tilde{\mathcal{S}}'\tilde{\rho}^{\gamma}
 \qquad \text{in}\ \Omega_{L},\\[2mm]
\psi=\frac{x_{2}-f_{1}(x_{1})}{f_{2}(x_{1})-f_{1}(x_{1})}m\qquad\qquad\,\,\,\, \text{on}\ \partial\Omega_{L},
\end{cases}
\eee
where $\Omega_{L}$ satisfies
$$
\{(x_{1},x_{2}) \,:\, (x_{1},x_{2})\in\Omega, |x_{1}|<L\}\subset\Omega_{L}
\subset\{(x_{1},x_{2})\,:\,(x_{1},x_{2})\in\Omega, |x_{1}|<4L\}
$$
for all positive constants $L>L_0>0$, with $L_{0}$ sufficiently large,
and $\partial\Omega_{L}\in C^{2,\alpha_{1}}, 0<\alpha_{1}<\alpha$, satisfies the uniform exterior sphere
condition with uniform radius $r_{0}$, $0<r_{0}<r$.

\medskip
2. Equation \eqref{3.10} can be written as \be \label{3.13}
\tilde{A}_{ij}\partial_{ij}\psi-\tilde{\rho}_{2}|\nabla\psi|^{2}=(\tilde{B}'-
\frac{1}{\gamma}\tilde{\mathcal{S}}'\tilde{\rho}^{\gamma-1})\tilde{\rho}^3,
\ee where the repeated index is the summation with respect to the
index from now on and \be
\tilde{\rho}_{2}=\frac{-2\tilde{\mathcal{B}}'\tilde{\rho}-(\gamma-1)\zeta'_0\tilde{S}'\tilde{\rho}^{\gamma}
+(\gamma+1)\tilde{\mathcal{S}}'\tilde{\rho}^{\gamma}}
{4\tilde{\mathcal{B}}-(\gamma+1)^2\mathcal{S}\rho^{\gamma-1}+(\gamma^2-1)\zeta'_{0}\tilde{S}\rho^{\gamma-1}}.
\ee Therefore, \eqref{3.13} becomes \be \label{3.15}
\tilde{S}_{ij}(\nabla\psi,\psi)\partial_{ij}\psi=\mathcal{F}(\nabla\psi,\psi),
\ee where
$$
\mathcal{F}(\nabla\psi,\psi)
=(\tilde{B}'-\frac{1}{\gamma}\tilde{\mathcal{S}}'\tilde{\rho}^{\gamma-1})\tilde{\rho}^3+\tilde{\rho}_{2}|\nabla\psi|^2.
$$
Instead of \eqref{3.12}, we first solve the following problem:
\bee\label{3.16}
\begin{cases}
 \tilde{A}_{ij}(\nabla\psi,\psi)\partial_{ij}\psi=\mathcal{\tilde{F}}(\nabla\psi,\psi)\qquad
    \text{in}\ \Omega_{L},\\[2mm]
\psi=\frac{x_{2}-f_{1}(x_{1})}{f_{2}(x_{1})-f_{1}(x_{1})}m\qquad\qquad\qquad\, \text{on}\ \partial\Omega_{L},
\end{cases}
\eee
where
$\mathcal{\tilde{F}}(\nabla\psi,\psi)
=(\tilde{B}'-\frac{1}{\gamma}\tilde{\mathcal{S}}'\tilde{\rho}^{\gamma-1})\tilde{\rho}^3
+\tilde{\rho}_{2}\tilde{\Delta}^2$
for dealing with the fact that $\mathcal{\tilde{F}}$ has quadratic
growth in $|\nabla\psi|$. By the definition of $\zeta$,
$\tilde{\mathcal{S}}$, and $\tilde{\mathcal{B}}$, we have
\be\label{3.17} |\tilde{\mathcal{F}}(\nabla\psi,\psi)|\leq C\delta.
\ee

3. Then, by the standard existence theory of elliptic equations,
there exists a solution $\psi_{L}$ to \eqref{3.16}.
Furthermore, writing $\psi^{-}_{L}=\min\{\psi_{L},0\}$ and $\psi^{+}_{L}=\max\{\psi_{L},0\}$,
by the maximum principle with the source term (cf. Theorem 3.7 in \cite{GT}),
\be \label{3.18}
\min\limits_{\partial\Omega_{L}}\psi^{-}_{L}-\frac{C}{\lambda}\sup\limits_{\Omega_{L}}|\tilde{\mathcal{F}}|
\leq\psi_{L}
\leq\sup\limits_{\partial\Omega_{L}}\psi^{+}_{L}+\frac{C}{\lambda}\sup\limits_{\Omega_{L}}|\tilde{\mathcal{F}}|,
\ee
where $C=e^{d}-1$ with $d=\sup\{f_{2}(x_{1})-f_{1}(x_{1})\}$. Then we have
$$
-C\delta^{1-2\beta}-C\delta^{1-\beta}\leq\psi_{k}\leq m+C\delta^{1-2\beta}+C\delta^{1-\beta}
\qquad \text{for k sufficiently large}.
$$
Moreover, one can obtain some other estimates for $\psi_{k}$.
In fact, we can use the following more precise form
with the same notations and symbols as those in Chapter 12 in \cite{GT},
\be\label{3.19}
[u]_{1,\alpha}\leq C(\gamma,\Omega)\Big(1+\|\nabla u\|_{0}+\frac{\|f\|_0}{\lambda}\Big).
\ee
Here, $C(\gamma,\Omega)$ depends only on $\text{diam}(\Omega)$ and the $C^{2}$--norm of $\partial\Omega$.

Applying estimate \eqref{3.19} to problem \eqref{3.16} deduces that
there exists $\mu=\mu(\frac{\Lambda}{\lambda})>0$ such that,
for any $x_{0}\in\bar{\Omega}_{L}$ and for $\psi_{k}$ with $k\geq 4L$, we have
\be\label{3.20}
[\psi_{k}]_{1,\mu;B_{1}(x_{0})\cap\Omega_{L}}
\leq C(\frac{\Lambda}{\lambda},\|f_1\|_2, \|f_2\|_{2})
\Big(1+\|\nabla \psi_{k}\|_{0;B_{1}(x_{0})\cap\Omega_{L}}
+\frac{\|\tilde{\mathcal{F}}\|_0}{\lambda}\Big).
\ee
Furthermore, using the interpolation inequality and the maximum principle \eqref{3.18},
we obtain
\begin{equation*}
\|\psi_{k}\|_{1;B_{1}(x_{0})\cap\Omega_{L}}
\leq\eta C(\frac{\Lambda}{\lambda},\|f_1\|_2, \|f_2\|_{2})\big(1+\|\nabla \psi_{k}\|_{0;B_{1}(x_{0})
\cap\Omega_{L}}+\frac{\|\tilde{\mathcal{F}}\|_0}{\lambda}\big)
+C_{\eta}\big(m+ \frac{\|\tilde{\mathcal{F}}\|_0}{\lambda}\big),
\end{equation*}
where $C>0$ is the same constants as that in \eqref{3.18}.
Taking $\eta_{0}$ sufficiently small so that
$\eta C(\frac{\Lambda}{\lambda},\|f_1\|_2,\|f_2\|_{2})\leq \frac{1}{2}$
if $\eta\leq \eta_{0}$, then
\be\label{3.21}
\|\psi_{k}\|_{1;B_{1}(x_{0})\cap\Omega_{L}}
\leq\eta C(\frac{\Lambda}{\lambda},\|f_1\|_2, \|f_2\|_{2})\big(1+\frac{\|\tilde{\mathcal{F}}\|_0}{\lambda}\big)
+C_{\eta}\big(m+\frac{\|\tilde{\mathcal{F}}\|_0}{\lambda}\big).
\ee

Thus, the H\"{o}lder estimate \eqref{3.20} becomes
\begin{equation}\label{3.22}
\begin{array}{lll}
&&\|\psi_{k}\|_{1,\mu;B_{1}(x_{0})\cap\Omega_{L}}\\[1.5mm]
&&=\|\psi_{k}\|_{1;B_{1}(x_{0})\cap\Omega_{L}}
+[\psi_{k}]_{1,\mu;B_{1}(x_{0})\cap\Omega_{L}}\\[1.5mm]
&&\leq \big(1+C(\frac{\Lambda}{\lambda},\|f_1\|_2, \|f_2\|_{2})\big)\|\psi_{k}\|_{1;B_{1}(x_{0})\cap\Omega_{L}}
+C(\frac{\Lambda}{\lambda},\|f_1\|_2, \|f_2\|_{2})\big(m+\frac{\|\tilde{\mathcal{F}}\|_0}{\lambda}\big)\\[1.5mm]
&&\leq C(\frac{\Lambda}{\lambda},\|f_1\|_2, \|f_2\|_{2})\big(1+m+\frac{\|\tilde{\mathcal{F}}\|_0}{\lambda}\big).
\end{array}
\end{equation}
Since, for any $x,y\in\bar{\Omega}_{L}$,
$$\frac{|\nabla\psi_{k}(x)-\nabla\psi_{k}(y)|}{|x-y|^{\mu}}\leq
\begin{cases}
  \|\psi_{k}\|_{1,\mu;B_{1}(x_{0})\cap\Omega_{L}}\qquad &\text{if}\ y\in B_{1}(x_{0})\cap\Omega_{L},\\[1.5mm]
 2\|\psi_{k}\|_{1;B_{1}(x_{0})\cap\Omega_{L}}\qquad&\text{if}\ y\notin B_{1}(x_{0})\cap\Omega_{L},
\end{cases}
$$
which, together with \eqref{3.21} and \eqref{3.22}, yields the following H\"{o}lder estimate:
\be\label{3.23}
[\psi_{k}]_{1,\mu;\Omega_{L}}
\leq C(\frac{\Lambda}{\lambda},\|f_1\|_2, \|f_2\|_{2})\big(1+m+\frac{\|\tilde{\mathcal{F}}\|_0}{\lambda}\big).
\ee
Thus, it follows from the standard Schauder estimate
that
$$
\|\psi_{k}\|_{2,\alpha;B_{1/2}(x_{0})\cap\Omega_{L}}
\leq C(\frac{\Lambda}{\lambda},\|f_1\|_{C^{2,\alpha}}, \|f_2\|_{C^{2,\alpha}}, m,\frac{\|\tilde{\mathcal{F}}\|_0}{\lambda}).
$$
Thus,
\be\label{3.24}
\|\psi_{k}\|_{2,\alpha;\Omega_{L}}
\leq C(\frac{\Lambda}{\lambda},\|f_1\|_{2,\alpha},\|f_2\|_{{2,\alpha}}, m,\frac{\|\tilde{\mathcal{F}}\|_0}{\lambda}).
\ee

4. Using the Arzela-Ascoli lemma and a diagonal procedure, we see that there exists a subsequence $\psi_{k_{l}}$
such that
$$
\psi_{k_{l}}\rightarrow\psi\ \text{in}\ C^{2,\vartheta}(K)\qquad
\text{for any compact set}\ K\subset\bar{\Omega}\ \text{and}\ \vartheta<\alpha.
$$
Here, $\psi$ satisfies the following problem:
$$
\begin{cases}
 \tilde{A}_{ij}(\nabla\psi,\psi)\partial_{ij}\psi=\mathcal{\tilde{F}}(\nabla\psi,\psi)\qquad
    \text{in}\ \Omega,\\[2mm]
\psi=\frac{x_{2}-f_{1}(x_{1})}{f_{2}(x_{1})-f_{1}(x_{1})}m\qquad\qquad\qquad\,\, \text{on}\ \partial\Omega,
\end{cases}
$$
with the estimate

\begin{equation}\label{3.25}
\begin{array}{lll}
\|\psi\|_{1;\Omega}&\leq\eta
C(\lambda,\|f_1\|_{2}, \|f_2\|_2)\big(1+\frac{\|\tilde{\mathcal{F}}\|_0}{\lambda}\big)
+C_{\eta}\big(m+\frac{\|\tilde{\mathcal{F}}\|_0}{\lambda}\big)\\
&\leq\eta C(\lambda,\|f_1\|_{2}, \|f_2\|_2)\big(1+C\delta\big)
+C_{\eta}\big(m+C\delta\big),
\end{array}
\end{equation}
where $\eta\in(0,\eta_{0})$ and $C$ depends only on
$\bar{\delta}_{0}$, $\bar{m}$, $\Lambda$, and $\lambda$.
Next, we prove that
$$
|\nabla\psi|^2\leq(\gamma-1)\mathcal{S}\rho^{\gamma+1}-2\epsilon.
$$
Otherwise,
\begin{eqnarray*}
(\gamma-1)\mathcal{S}\rho^{\gamma+1}&<&|\nabla\psi|^2+2\epsilon\\
&\leq&\eta\, C(\lambda, \|f_i\|_2)(1+C\delta)+C_{\eta}(m+C\delta)\\
&\leq& (\gamma-1)\mathcal{S}\big(\frac{2\mathcal{B}}{(\gamma+1)\mathcal{S}}\Big)^{\frac{\gamma+1}{\gamma-1}}.
\end{eqnarray*}
Thus,
$$
\rho<\Big(\frac{2\mathcal{B}}{(\gamma+1)\mathcal{S}}\Big)^{\frac{1}{\gamma-1}}
$$
and
$$
2\mathcal{B}\rho^2-(\gamma+1)\mathcal{S}\rho^{\gamma+1}\geq0,
$$
which contradict with the fact that
$$
2\mathcal{B}\rho^2-(\gamma+1)\mathcal{S}\rho^{\gamma+1}=\zeta_0(|\nabla\psi|^2-(\gamma-1)\mathcal{S}\rho^{\gamma+1})<0.
$$
Thus, the solution $\psi$ satisfies
\be\label{3.26}
|\nabla\psi|^{2}\leq \underline{\Sigma}(\epsilon)-2\epsilon
\ee
for
any $\delta\in(0,\delta_{1})$ and
$m\in(\delta^{\beta},2\delta^{\beta/2}_{1})$.  Then \eqref{3.11}
follows from \eqref{3.25} and \eqref{3.26}.

Furthermore, \eqref{3.23} and \eqref{3.24} yield the following higher order estimates
\be
\|\psi\|_{1,\mu;\bar{\Omega}}
\leq C(\frac{\Lambda}{\lambda},\|f_1\|_{2}, \|f_2\|_2)\big(1+m+\frac{\|\mathcal{F}\|_0}{\lambda}\big),
\ee
and
\be\label{3.28}
\|\psi\|_{2,\bar{\Omega}}\leq C(\frac{\Lambda}{\lambda},\|f_1\|_{{2,\alpha}}, \|f_2\|_{2,\alpha}, m,
\frac{\|\mathcal{F}\|_0}{\lambda}).
\ee
This completes the proof.
\end{proof}

\begin{remark} Estimate \eqref{3.11} in Proposition \ref{1} implies that
the cut--off function introduced in \eqref{cut-off} and
\eqref{cut-off using} can be removed.
\end{remark}

\section{Far Field Behavior of Solutions of {\bf Problem 3}}

In this section, we study the far field behavior of solutions to
{\it Problem 3}. We now show that the solutions to {\it Problem 3}
satisfy the asymptotic behavior \eqref{1.18}--\eqref{1.24}, and
$0\leq\psi\leq m$. From this, we can remove both the extension and
the elliptic cut--off \eqref{3.10}. Therefore, these solutions solve
{\it Problem 2}. In addition, the stream function formulation is
consistent with the formulation of {\it Problem 1} for the
non-isentropic  Euler system in the infinitely long nozzle, as long
as the flow induced by a solution to {\it Problem 2}  satisfies
\eqref{1.18}--\eqref{1.24} and \eqref{2.10}. Furthermore, the far
field behavior is crucial also for the consequent result of the
uniqueness of the solutions. First we have

\begin{lemma} \label{3}
For $\epsilon>0$, there exists $\delta_{2}\in (0,\bar{\delta}_{0}]$
such that, if
\begin{itemize}
\item[(i)] $\|(S-\underline{S},B-\underline{B})\|_{C^{1,1}}\leq\delta\leq\delta_{2}$;

\medskip
\item[(ii)] $m\in(\delta^{\beta},\bar{m})$, where $\bar{m}$ is defined as Proposition {\rm \ref{2}},

\end{itemize}
\medskip
then there exists a function $\bar{\psi}$ that satisfies
$$
\psi\rightarrow\bar{\psi} \qquad\text{as }x_1\rightarrow-\infty,
$$
and
 \be\label{4.6}
 \bar{\psi}(x_{1},x_{2})=\bar{\psi}(x_{2})=\int_{0}^{x_{2}}\rho_{0}(s;p_0)u_{0}(s)ds,
 \ee
where $\rho_{0}$ and $u_{0}$ are uniquely determined by $S$, $B$,
and $m$ as \S 2, so $\bar{\psi}$ is independent of $x_{1}$.
\end{lemma}

\begin{proof}
The proof is based on the blowup argument in combination with the energy estimate,
which consists of three parts:
The first is for the existence of $\bar{\psi}$, the second is for
the independence of $\bar{\psi}$ of $x_{1}$, and the third is
the explicit form \eqref{4.6} for $\bar{\psi}$.

\smallskip
1. {\it Existence of the far field function}. It is convenient to
introduce a new coordinate to flatten the boundary walls of the
nozzle, as follows:
$$
\begin{cases}
t_1(x_{1},x_{2})=x_1,\\
t_2(x_{1},x_{2})=\frac{x_2-f_1(x_1)}{f_2(x_1)-f_1(x_1)},
\end{cases}
$$
then the nozzle becomes $(-\infty,\infty)\times[0,1]$. Obviously,
the coordinate transform is reversible, since
$$
 \mbox{det}\left[\begin{array}{l}
\frac{\partial t_1}{\partial x_1}, \frac{\partial t_1}{\partial
x_2}\\[2mm]
\frac{\partial t_2}{\partial x_1}, \frac{\partial t_2}{\partial x_2}
 \end{array}\right] =
 \left| \left[\begin{array}{l}
1,\qquad 0\\[2mm]
*, \frac{1}{f_2(x_1)-f_1(x_1)}
 \end{array}\right]\right| = \frac{1}{f_2(x_1)-f_1(x_1)}\neq0.
$$
In addition, we remark here that the equation
does not change the type of ellipticity under the
coordinate transformation, since
\begin{eqnarray*}
&&a_{ij}\partial_{x_ix_j}\psi\\
&&=a_{11}\partial_{t_1t_1}\psi
   -2\Big(\frac{a_{11}\big(x_2-f_1(x_1)\big)\big(f_2'(x_1)-f_1'(x_1)\big)}{\big(f_2(x_1)-f_1(x_1)\big)^2}
   -\frac{a_{12}}{\big(f_2(x_1)-f_1(x_1)\big)}\Big)\partial_{t_1t_2}\psi\\
&&\quad+\Big(\frac{a_{11}\big(x_2-f_1(x_1)\big)^2\big)(f_2'(x_1)-f_1'(x_1)\big)^2}{\big(f_2(x_1)-f_1(x_1)\big)^4}
-\frac{2a_{12}\big(x_2-f_1(x_1)\big)\big(f_2'(x_1)-f_1'(x_1)\big)}{\big(f_2(x_1)-f_1(x_1)\big)^3}\\
&&\qquad\,\,\, +\frac{a_{22}}{\big(f_2(x_1)-f_1(x_1)\big)^2}\Big)\partial_{t_2t_2}\psi\\
&&\quad +\, \text{lower terms (involving $\partial_{t_i}\psi$ and $\psi$)}.
\end{eqnarray*}
In the new coordinates, define
$$
\psi^{(n)}=\psi(t_1(x_{1}-n,x_{2}), t_2(x_{1}-n,x_2)).
$$
For any compact set $K\subset(-\infty,\infty)\times[0,1]$, it
follows from \eqref{3.28} and the $C^{2,\a}$-bounds of the walls
$f_1$ and $f_2$ that
$$
||\psi^{(n)}||_{C^{2,\a}(K)}\leq C\qquad\mbox{for $n$ sufficiently
large.}
$$
Then, as in Step 4 of the proof to Proposition \ref{1},
there exists a subsequence $\psi^{(n_{k})}$ such that
 \be\label{4.1}
\psi^{(n_{k})}\rightarrow\bar{\psi}\qquad \text{in}\
C^{2,\vartheta}(K)
\ee
for any compact set $K\subset(-\infty,\infty)\times[0,1]$ and any
$\vartheta\in(0,\alpha)$. From \eqref{1.3}--\eqref{1.6} and
\eqref{3.28}, and the facts that $f_1(x_1)\rightarrow 0$
and $f_2(x_1)\rightarrow 1$ in $C^{2,\a}$ as
$x_1\rightarrow-\infty$, which also means that
$f'_i(x_1)\rightarrow0$ in $C^{1,\a}$ as $x_1\rightarrow-\infty$,
then $\bar{\psi}$ satisfies
 \bee\label{4.2}
\begin{cases}
\nabla\cdot(\frac{\nabla\bar{\psi}}{\tilde{\rho}(|\nabla\bar{\psi}|^{2},\bar{\psi})})
=(\tilde{B}'\tilde{\rho}
-\frac{1}{\gamma}\tilde{S}'\tilde{\rho}^{\gamma})(|\nabla\bar{\psi}|^{2},\bar{\psi})\qquad \text{in}\ D,\\[2mm]
\bar{\psi}=0\qquad \,\, \text{on}\ x_{2}=0,\\[2mm]
\bar{\psi}=m\qquad \text{on}\ x_{2}=1,
\end{cases}
 \eee
where $D=(-\infty,\infty)\times(0,1)$, and $\bar{\psi}$ also satisfies
 \be\label{4.5}
|\bar{\psi}|\leq C(\epsilon,\delta),\qquad
|\nabla\bar{\psi}|^{2}\leq \underline{\Sigma}(\epsilon)-2\epsilon.
 \ee
Thus, by similar arguments as in \S 3, on any compact set $E\subset (-\infty,\infty)\times[0,1]$,
$$
\|\bar{\psi}\|_{C^{1,\mu}(E)}\leq C(\epsilon,\delta),
$$
and
\be\label{4.4}
\|\bar{\psi}\|_{C^{2,\alpha}(E)}\leq C(\epsilon,\delta).
\ee
Thus, $\bar{\psi}\in C^{2,\alpha}(\bar{D})$. This completes the first part.

2. Differentiate the equation in \eqref{4.2} with respect to $x_{1}$
and set $\omega=\bar{\psi}_{x_{1}}$. Then
\begin{equation}\label{4.7}
\partial_{i}\Big(\frac{\tilde{A}_{ij}(\nabla\bar{\psi},\bar{\psi})}{\tilde{\rho}^{2}(|\nabla\bar{\psi}|^{2},\bar{\psi})}
\partial_{j}\omega\Big)
-\partial_{i}\Big(\frac{\tilde{\rho}_{2}(|\nabla\bar{\psi}|^{2},\bar{\psi})\partial_{i}\bar{\psi}}
 {\tilde{\rho}^{2}(|\nabla\bar{\psi}|^{2},\bar{\psi})}\omega\Big)
=\tilde{\Theta}(|\nabla\tilde{\psi}|^{2},\bar{\psi})\omega
+\tilde{\vartheta}(|\nabla\tilde{\psi}|^{2},\bar{\psi})\partial_{i}\bar{\psi}\partial_{i}\omega,
\end{equation}
where $\tilde{A}_{ij}(q,z)$, $\tilde{\Theta}(q,z)$, and
$\tilde{\vartheta}(q,z)$ satisfy
\begin{eqnarray*}
&&\tilde{A}_{ij}(q,z)=\tilde{\rho}(|q|^{2},z)\delta_{ij}-2\tilde{\rho}_{1}(|q|^{2},z)q_{i}q_{j},\label{4.8}\\
&&\tilde{\Theta}(s,z)=\tilde{\mathcal{B}}''(z)\tilde{\rho}(s,z)-\frac{1}{\gamma}\tilde{\mathcal{S}}''(z)
\tilde{\rho}^{\gamma}+\big(\tilde{\mathcal{B}}'(z)-\tilde{\mathcal{S}}'(z)
\tilde{\rho}^{\gamma-1}(s,z)\big)\tilde{\rho}_{2}(s,z),\label{4.9}\\
&&\tilde{\vartheta}(s,z)=2\big(\tilde{\mathcal{B}}'(z)-\tilde{\mathcal{S}}'(z)
\tilde{\rho}^{\gamma-1}(s,z)\big)\tilde{\rho}_{1}(s,z), \label{4.10}
\end{eqnarray*}
for $q\in\mathbb{R}^{2}$, $s\geq0$, and $z\in\mathbb{R}$, where
$(\tilde{S}, \tilde{B})\in C^{1,1}(\mathbb{R})$. Since it is unknown
whether $\bar{\psi}\in C^{3}(D)$, equation \eqref{4.7} holds in the
weak sense. It follows from \eqref{4.5} that
$$
|\tilde{A}_{ij}(\nabla\bar{\psi},\bar{\psi})|\leq\Lambda(\epsilon),
$$
where $\Lambda$ depends only on $\epsilon$. Furthermore, $\omega$
satisfies the following boundary conditions:
$$
\omega=0\qquad \text{on}\ x_{2}=0,1.
$$
As usual for energy estimates, let $\eta$ be a $C^{\infty}_{0}$--function
satisfying
\be\label{4.11}
\eta=1\ \text{for}\ |s|<L,\qquad
\eta=0\ \text{for}\ |s|>L+1,\qquad \ |\eta'(s)|\leq2.
\ee
Multiply
$\eta^{2}(x_{1})\omega$ and integrate it on both sides of
\eqref{4.7}, then integrate the left side, and plug the
explicit forms of $\tilde{A}_{ij}$,
$\tilde{\r}_1(|\nabla\bar{\psi}|^2,\bar{\psi})$, and
$\tilde{\r}_2(|\nabla\bar{\psi}|^2,\bar{\psi})$ into it. We obtain
\begin{eqnarray}\label{4.8a}
\iint_{D}\frac{\eta^{2}|\nabla\omega|^{2}}{\tilde{\rho}
(|\nabla\bar{\psi}|^{2},\bar{\psi})}dx_{1}dx_{2}
=\sum\limits_{i=1}^{6}I_{i},
\end{eqnarray}
where
\begin{eqnarray*}
&&I_1=-\iint_{D}\frac{|\nabla\psi\cdot\nabla\omega|^{2}\eta^{2}}
{\tilde{\rho}(|\nabla\bar{\psi}|^{2},\bar{\psi})
\big(\tilde{\rho}^{2}(|\nabla\bar{\psi}|^{2},\bar{\psi})c^{2}-|\nabla\bar{\psi}|^{2}\big)}
dx_{1}dx_{2}, \nonumber\\[1.5mm]
&&I_2=-2\iint_{D}\frac{\tilde{A}_{ij}(\nabla\bar{\psi},\bar{\psi})}{\tilde{\rho}^{2}
(|\nabla\bar{\psi}|^{2},\bar{\psi})}\eta\omega \partial_{j}\omega\partial_{i}\eta dx_{1}dx_{2},\\[1.5mm]
&&I_3=\iint_{D}\frac{\tilde{\rho}_{2}(|\nabla\bar{\psi}|^{2},\bar{\psi})
   \nabla\bar{\psi}\cdot\nabla\eta}{\tilde{\rho}^{2}
(|\nabla\bar{\psi}|^{2},\bar{\psi})}\eta\omega^{2} dx_{1}dx_{2},\nonumber\\[1.5mm]
&&I_4=2 \iint_{D}\frac{\big(\tilde{\mathcal{B}}'(\bar{\psi})-\tilde{\mathcal{S}}'(\bar{\psi})
\tilde{\rho}^{\gamma-1}(s,\bar{\psi})\big)\nabla\psi\cdot\nabla\eta}
{\tilde{\rho}^{2}(|\nabla\bar{\psi}|^{2},\bar{\psi})c^{2}-|\nabla\bar{\psi}|^{2}}\eta\omega^{2} dx_{1}dx_{2},
  \nonumber\\[1.5mm]
&&I_5=-\iint_{D}\big(\tilde{\mathcal{B}}''(\bar{\psi})\tilde{\rho}(s,\bar{\psi})
-\frac{1}{\gamma}\tilde{\mathcal{S}}''(\bar{\psi})
\tilde{\rho}^{\gamma}\big)\eta^{2}\omega^{2} dx_{1}dx_{2},\nonumber\\[1.5mm]
&&I_6=-\iint_{D}\frac{\tilde{\rho}^2\big(\tilde{\mathcal{B}}'(\bar{\psi})-\tilde{\mathcal{S}}'(\bar{\psi})
\tilde{\rho}^{\gamma-1}(|\nabla\bar{\psi}|^{2},\bar{\psi})\big)^2}{\tilde{\rho}^{2}
 (|\nabla\bar{\psi}|^{2},\bar{\psi})c^{2}
 -|\nabla\bar{\psi}|^{2}}\tilde{\rho}_{2}(|\nabla\bar{\psi}|^{2},\bar{\psi})
\eta^{2}\omega^{2}  dx_{1}dx_{2}.\nonumber
\end{eqnarray*}

Now we make the estimates.  First, by the H\"{o}lder inequality, it is easy to
see that
$$
I_{1}+I_{4}+I_{6}\leq0.
$$
Second, since $\|(S-\underline{S}, B-\underline{B})\|_{C^{1,1}([0,1])}\leq \delta$
and $m\in(\delta^{\beta},\bar{m})$, we have
$$
\|(\tilde{\mathcal{S}}-\tilde{\mathcal{S}}(0),
\tilde{\mathcal{B}}-\tilde{\mathcal{B}}(0))\|_{C^{1,1}(\mathbb{R})}\leq \delta^{1-\beta}.
$$
Thus,
\be\label{4.13}
|I_{5}|\leq C\delta^{1-\beta}\int_{-L-1}^{L+1}\int_{0}^{1}\omega^{2}dx_{1}dx_{2},
\ee
and $\tilde{\rho}\leq\bar{\varrho}(\bar{D};x_2)$, where $C$ is
independent of $\epsilon$. Thus, from \eqref{4.8a} and the
definition of $\eta$, if $\delta_{2}$ is sufficiently small, we
obtain
\begin{eqnarray*}
&&\int_{-L}^{L}\int_{0}^{1}|\nabla \omega|^{2}dx_{2}dx_1\\
&&\leq |I_2+I_3|+|I_5|\\
&&\leq C(\epsilon)\Big(\int_{-L-1}^{-L}
+\int_{L}^{L+l}\Big)\Big(\int_{0}^{1}(|\nabla\omega|^{2}+|\nabla\omega|+\omega^2)dx_{2}\Big)dx_1\\
&& \quad +C\delta^{1-\beta}_2\int_{-L}^{L}\int_{0}^{1}|\nabla \omega|^{2}dx_{2}dx_{1}\qquad\\
&&\leq C(\epsilon)\Big(\int_{-L-1}^{-L}+\int_{L}^{L+l}\Big)\Big(\int_{0}^{1}(|\nabla\omega|^{2}+\omega^2)dx_{2}\Big)dx_1\\
&&\quad +\frac{1}{2}\int_{-L}^{L}\int_{0}^{1}|\nabla
\omega|^{2}dx_{2}dx_{1}.
\end{eqnarray*}
Notice that $\omega=0$ on $x_2=0,1$. It follows from the
Poincar\'{e} inequality that
there exists a constant $C$
independent of $l$ such that \be\label{4.16}
\int_{-L}^{L}\int_{0}^{1}|\nabla \omega|^{2}dx_{2}dx_{1} \leq
C\Big(\int_{-L-1}^{-L}+\int_{L}^{L+l}\Big)
\Big(\int_{0}^{1}|\nabla \omega|^{2}dx_{2}\Big)dx_1
\ee
for large $L$. It follows
from \eqref{4.4} that
$$
\int_{-L}^{L}\int_{0}^{1}|\nabla\omega|^{2}dx_{2}dx_1\leq
\Big(\int_{-L-1}^{-L}+\int_{L}^{L+l}\Big)\Big(\int_{0}^{1}|\nabla\omega|^{2}dx_{2}\Big)dx_1\leq
C
$$
for some uniform constant $C$ independent of $L$ and
for some constant $C$. Passing the limit $L\rightarrow\infty$  yields
$$
\int_{-\infty}^{\infty}\int_{0}^{1}|\nabla\omega|^{2}dx_{2}dx_1\leq
C.
$$
Hence,
\be
\Big(\int_{-L-1}^{-L}+\int_{L}^{L+l}\Big)
\Big(\int_{0}^{1}|\nabla\omega|^{2}dx_{2}\Big)dx_1\rightarrow0\qquad \text{as}\
L\rightarrow\infty.
\ee
Using \eqref{4.16} by passing the limit
$l\rightarrow\infty$ as before again, we obtain
$$
\int_{-\infty}^{\infty}\int_{0}^{1}|\nabla\omega|^{2}dx_{2}dx_{1}=0,
$$
which implies $\omega=0$.
Therefore,
$$
\bar{\psi}=\bar{\psi}(x_{2}),
$$
which solves the following boundary value problem: \bee\label{4.18}
\begin{cases}
 \frac{d}{dx_{2}}(\frac{\nabla\bar{\psi}}{\tilde{\rho}(|\nabla\bar{\psi}|^{2},\bar{\psi})})
=\tilde{G}(\nabla\bar{\psi},\bar{\psi})\qquad \text{in}\ D,\\[2mm]
\bar{\psi}(0)=0,\\[2mm]
\bar{\psi}(1)=m,
\end{cases}
\eee
which completes the first part.

\medskip
3. {\it Explicit form of $\bar{\psi}(x_2)$.} Suppose that there are two solutions
$\bar{\psi}_{1}$ and $\bar{\psi}_{2}$ to \eqref{4.18}. Let
$\bar{\phi}=\bar{\psi}_{1}-\bar{\psi}_{2}$. Then $\bar{\phi}$
satisfies \bee\label{4.19}
\begin{cases}
(\bar{a}\bar{\phi}'+\bar{b}\bar{\phi})'
=\bar{c}\bar{\phi}'+\bar{d}\bar{\phi},\\
\bar{\phi}(0)=\bar{\phi}(1)=0,
\end{cases}
\eee
where
\begin{eqnarray*}
&&\bar{a}=\int_{0}^{1}\frac{\tilde{\rho}(|\tilde{\psi}'|^{2},\tilde{\psi})
  -2\tilde{\rho}_{1}(|\tilde{\psi}'|^{2},\tilde{\psi})|\tilde{\psi}'|^{2}}{\tilde{\rho}^{2}
   (|\tilde{\psi}'|^{2},\tilde{\psi})}ds,\ \quad\bar{b}=\int_{0}^{1}\frac{-\tilde{\rho}_{2}(|\tilde{\psi}'|^{2},
   \tilde{\psi})\tilde{\psi}'}{\tilde{\rho}^{2}(|\tilde{\psi}'|^{2},\tilde{\psi})}ds,\\
&& \bar{c}=\int_{0}^{1}\tilde{\vartheta}(|\tilde{\psi}'|^{2},\tilde{\psi}')\tilde{\psi}'ds,\
\qquad\qquad\qquad\qquad\,\,\,\, \bar{d}=\int_{0}^{1}\tilde{\Theta}(|\tilde{\psi}'|^{2},\tilde{\psi}')ds,
\end{eqnarray*}
with $\tilde{\psi}=s\bar{\psi}_{1}+(1-s)\bar{\psi}_{2}$, where
$\tilde{\vartheta}$ and $\tilde{\Theta}$ are defined in \eqref{4.7}.
Multiplying $\bar{\phi}$ on both sides of equation \eqref{4.19} and
integrating it over $[0,1]$, we have
$$
-\int_{0}^{1}(\bar{a}\bar{\phi}'^{2}+\bar{b}\bar{\phi}'\bar{\phi})dx_{2}=
\int_{0}^{1}(\bar{c}\bar{\phi}'\bar{\phi}+\bar{d}\bar{\phi}^{2})dx_{2}.
$$
Thus, we have
\begin{eqnarray*}
&&\int_{0}^{1}\int_{0}^{1}\frac{\bar{\phi}'^{2}}{\tilde{\rho}^{2}
(|\tilde{\psi}'|^{2},\tilde{\psi})}dsdx_{2}\\[1.5mm]
&&=-\int_{0}^{1}\int_{0}^{1}\big(\tilde{\mathcal{B}}''(\tilde{\psi})\tilde{\rho}(|\tilde{\psi}'|^{2},\tilde{\psi})
-\frac{1}{\gamma}\tilde{\mathcal{S}}''(\tilde{\psi})
\tilde{\rho}^{\gamma}(\tilde{\psi}'^{2},\tilde{\psi})\big)\bar{\phi}^{2}dsdx_{2}\\[1.5mm]
&&\quad+\int_{0}^{1}\int_{0}^{1}\frac{\tilde{\rho}^2|\tilde{\psi}'|^{2}\bar{\phi}'^{2}}
{2\tilde{\mathcal{B}}\tilde{\rho}-(\gamma+1)
\tilde{\mathcal{S}}\tilde{\rho}^{\gamma}}dsdx_2
+\int_{0}^{1}\int_{0}^{1}\frac{(\tilde{\mathcal{B}}'\tilde{\rho}-\tilde{\mathcal{S}}'
\tilde{\rho}^{\gamma})^2\bar{\phi}^2}{2\tilde{\mathcal{B}}\tilde{\rho}-(\gamma+1)
\tilde{\mathcal{S}}\tilde{\rho}^{\gamma}}dsdx_2\\[1.5mm]
&&\quad-\int_{0}^{1}\int_{0}^{1}\frac{2(\tilde{\mathcal{B}}'\tilde{\rho}-\tilde{\mathcal{S}}'
\tilde{\rho}^{\gamma})\tilde{\psi}'\bar{\phi}'\bar{\phi}}{2\tilde{\mathcal{B}}\tilde{\rho}^2-(\gamma+1)
\tilde{\mathcal{S}}\tilde{\rho}^{\gamma+1}}dsdx_2.
\end{eqnarray*}
The sum of the last three terms is negative. By the smallness of
$\delta$ and the Poincar\'{e} inequality as in Step $2$, we
have
$$
\int_{0}^{1}|\bar{\phi}'|^{2}dx_{2}\leq0,
$$
which yields $\bar{\phi}=0$. Thus, the solution to \eqref{4.18} is
unique. Obviously, we know that
$$
\bar{\psi}=\bar{\psi}(x_{2})=\int_{0}^{x_{2}}\rho_{0}(s;p_0)u_{0}(s)ds.
$$
is a solution to the boundary value problem \eqref{4.18}. In fact,
from \eqref{2.18},
$$
\frac{u^{2}_{0}(x_{2})}{2}+\mathcal{S}(\bar{\psi}(x_2))\rho^{\gamma-1}_0(x_2;p_0)=\mathcal{B}(\bar{\psi}(x_{2})),\qquad
\rho_0(x_2;p_0)=\Big(\frac{\gamma
p_0}{(\gamma-1)\mathcal{S}(\bar{\psi}(x_2))}\Big)^{\frac{1}{\gamma}},
$$
we have
$$
\left(\frac{\bar{\psi}_{x_2}(x_2)}{\rho_0(x_2;p_0)}\right)_{x_2}=u'_0(x_2)=\mathcal{B}'(\bar{\psi})\rho_0(x_2;p_0)
-\frac{1}{\gamma}\mathcal{S}'(\bar{\psi})\rho_0^{\gamma}(x_2;p_0).
$$
 This completes the proof.
\end{proof}

It follows from Lemma \ref{3} that the flow induced by the
stream function $\psi$ satisfies \eqref{1.18}--\eqref{1.19} in the
upstream. The similar properties in the downstream can be obtained
in the same way.

As indicated at the beginning of this section, an important maximum
estimate for the stream function can be yielded as a consequence of
the far field behaviors, and one could see the reason for the way
defining $\tilde{\mathcal{B}}$ in \eqref{3.0}.

\begin{proposition}\label{4}
Suppose that \eqref{1.13} holds with $\delta\leq
\min\{\delta_{1},\delta_{2}\}$, and $(S(x_2), B(x_2))$ satisfy
\eqref{1.16a}. Then the solution to {\rm Problem 3} satisfies
\be
0\leq\psi \leq m\qquad\text{in}\ \Omega.
\ee
\end{proposition}

\begin{proof} It follows from Proposition \ref{3} that
$$
\psi(x_{1},x_{2})\rightarrow
\int_{0}^{x_{2}}\rho_{0}(s;p_0)u_{0}(s)ds\qquad \text{uniformly as}\
x_{1}\rightarrow-\infty,
$$
and
$$
\psi(x_{1},x_{2})\rightarrow
\int_{0}^{x_{2}}\rho_{1}(s;p_1)u_{1}(s)ds\qquad \text{uniformly as}\
x_{1}\rightarrow \infty.
$$
Therefore, for any $\epsilon>0$, there exists $L>0$ such that
\be\label{4.20}
-\epsilon\leq\psi(x_{1},x_{2})<m+\epsilon\qquad \text{if}\ |x_{1}|\geq L.
\ee
We claim that
\be
-\epsilon\leq\psi(x_{1},x_{2})<m+\epsilon\qquad \text{if}\ |x_{1}|\leq L.
\ee

We show our claim by contradiction. More precisely, suppose that
there exists a point $X_{\max}=(x_{10},x_{20})$ with $|x_{10}|\leq
L$ such that
$$
\psi(X_{\max})=\max\limits_{X\in\{|x_{1}|\leq L\}}\psi(x_{1},x_{2})\geq m+\epsilon.
$$
Let $\hat{\rho}=\tilde{\rho}(\psi(X_{\max}))$. We  have
\begin{equation*}
\begin{array}{lll}\tilde{A}_{ij}(0,\psi(X_{\max}))\partial_{ij}\psi(X_{\max})&=&\hat{\rho}
(\mathcal{B}'-\frac{1}{\gamma}\mathcal{S}'\hat{\rho}^{\gamma-1})\\[2mm]
&=&\hat{\rho}
(\mathcal{B}'-\frac{1}{\gamma}\frac{\mathcal{B}}{\mathcal{S}}\mathcal{S}')(\psi(X_{\max}))\\[2mm]
&=&\frac{\hat{\rho}}{\mathcal{B}(\psi(X_{\max}))}(\ln(\mathcal{S}^{-\gamma}\mathcal{B}))'(\psi(X_{\max})))\geq0,
\end{array}
\end{equation*}
where we have used the fact that $\nabla\psi(X_{\max})=0$,
hence $\hat{\rho}^{\gamma-1}=\frac{\mathcal{B}}{\mathcal{S}}(\psi(X_{\max}))$.
Thus, we have
$$
0\leq\tilde{A}_{ij}(0,\psi(X_{\max}))\partial_{ij}\psi(X_{\max})<0,
$$
which is a contradiction. That is,
$$
-\epsilon\leq\psi(x_{1},x_{2})<m+\epsilon\qquad \text{in}\ \{\psi\geq m\}\cap\{|x_{1}|\leq L\}.
$$
Since $\frac{d}{d\psi}\ln(\mathcal{S}^{-\gamma}\mathcal{B})\leq0$ in
the domain $\{\psi\leq 0\}$, we can similarly show that
$$
-\epsilon\leq\psi(x_{1},x_{2})<m+\epsilon\qquad \text{in}\ \{\psi\leq0\}\cap\{|x_{1}|\leq L\}.
$$
Combining these estimates together, we obtain
$$
-\epsilon\leq\psi(x_{1},x_{2})<m+\epsilon\qquad \text{in}\ \Omega.
$$
Since $\epsilon$ is arbitrary, we have
$$
0\leq\psi(x_{1},x_{2})\leq m\qquad \text{in}\ \Omega.
$$
This completes the proof.
\end{proof}

\section{Existence and Uniqueness of Solutions of {\bf Problem 2}}

Propositions \ref{1} and \ref{4} imply that the solutions
established in Proposition \ref{1} are the solutions of the boundary
value problem \eqref{2.33}, {\it Problem 2}.

\begin{proposition}
Let the boundary $\partial\Omega$ satisfy \eqref{1.3}--\eqref{1.6}.
Let \eqref{1.13} hold with $\delta\leq \min\{\delta_{1},\delta_{2}\}$,
and $(S(x_2), B(x_2))$ satisfy \eqref{1.16a}.
Then there exists $0<\delta_1\leq\min\{\bar{\delta}_0,\bar{\delta}_1\}$ such that,
if
$\|(S-\underline{S}, B-\underline{B})\|_{C^{1,1}([0,1])}\leq\delta$ with $0<\delta\leq \delta_1$
and $m\in(\delta^{\beta},m_1)$ with $m_1=2\delta^{\beta/2}\leq\bar{m}$,
where $\bar{m}$ is defined in {\rm Proposition 2.2},
then {\rm Problem 2} (i.e. \eqref{2.33}) has a uniformly subsonic
solution $\psi\in C^{2,\alpha}(\overline{\Omega})$ satisfying
$$
 0\leq\psi \leq m\qquad\text{in}\ \Omega.
$$
\end{proposition}

Next, we will use the energy estimates again to show that uniformly
subsonic solutions of {\it Problem 2} are unique.

\begin{proposition}\label{6} Let the boundary $\partial\Omega$ satisfy \eqref{1.3}--\eqref{1.6}.
Then there exists $\delta_{3}\in(0,\bar{\delta}_{0}]$ such that, if
\begin{itemize}
\item[(i)]\ $\|(S-\underline{S}, B-\underline{B})\|_{C^{1,1}([0,1])}\leq\delta$ with $0<\delta\leq \delta_{3}$,

\medskip
\item[(ii)]\ $m\in(\delta^{\beta},\bar{m})$,
\end{itemize}
then there exists at most one solution $\psi$ of {\rm Problem 2}
satisfying
\be\label{4.23}
0\leq\psi(x_{1},x_{2})\leq m,\qquad
|\nabla\psi|^{2}\leq\underline{\Sigma}(\epsilon)-2\epsilon\,\,\,\,
\text{for some}\ \epsilon>0.
\ee
\end{proposition}

\begin{proof} As before, let $\psi_{1}$ and $\psi_{2}$ be two solutions to \eqref{2.33}.
Set $\hat{\psi}=\psi_{1}-\psi_{2}$. Then $\hat{\psi}$ satisfies
\bee\label{4.24}
\begin{cases}
\partial_{i}(a_{ij}\partial_{j}\hat{\psi})+\partial_{i}(b_{i}\hat{\psi})=c_{i}\partial_{i}\hat{\psi}+d\hat{\psi}
\qquad\mbox{in $\Omega$},\\[1.5mm]
\hat{\psi}=0\qquad \text{on}\ W_{1}\cup W_{2},
\end{cases}
\eee
where
\begin{eqnarray*}
&&a_{ij}=\int_{0}^{1}\frac{A_{ij}(D\tilde{\psi},\tilde{\psi})}{\rho^{2}(|\nabla\tilde{\psi}|^{2},\tilde{\psi})}ds,
\qquad\,\,  b_{i}=-\int_{0}^{1}\frac{\rho_{2}(|\nabla\tilde{\psi}|^{2},\tilde{\psi})\partial_{i}\tilde{\psi}}
{\rho^{2}(|\nabla\tilde{\psi}|^{2},\tilde{\psi})}ds,\\
&& c_{i}=\int_{0}^{1}\vartheta(|\nabla\tilde{\psi}|^{2},\tilde{\psi})\partial_{i}\tilde{\psi}ds,
\qquad
 d=\int_{0}^{1}\Theta(|\nabla\tilde{\psi}|^{2},\tilde{\psi})ds,
\end{eqnarray*}
 $\tilde{\psi}=s\psi_{1}+(1-s)\psi_{2}$, $A_{ij}$, $\Theta$ and $\vartheta$ are defined as \eqref{4.7},
 except we replace $(\tilde{\mathcal{S}}, \tilde{\mathcal{B}}, \tilde{\rho})$
 by $(\mathcal{S}, \mathcal{B}, \rho)$.

Multiplying $\eta^{2}\hat{\psi}^{+}$ and integrating on both sides of
\eqref{4.24}, where $\eta$ is defined in \eqref{4.11} and
$\hat{\psi}^{+}=\max\{\hat{\psi}(x),0\}$, then, similar to the proof of Lemma
\ref{3}, we have
$$
\iint_{\Omega\cap\{|x_{1}|\leq l\}\cap\{\hat{\psi}\geq0\}}|\nabla\hat{\psi}|^{2}dx_{1}dx_{2}
\leq C(\underline{B},\epsilon)\iint_{\Omega\cap\{l\leq|x_{1}|\leq l+1\}\cap\{\hat{\psi}\geq0\}}
|\nabla\hat{\psi}|^{2}dx_{1}dx_{2}.
$$
Since the solutions $\psi_{1}$ and $\psi_{2}$ have the same far
field behavior, and note that $|\hat{\psi}|$ and
$|\nabla\hat{\psi}|\rightarrow0$ as $|x_{1}|\rightarrow\infty$, we have
$$
\iint_{\Omega\cap\{\hat{\psi}\geq0\}}|\nabla\hat{\psi}|^{2}dx_{1}dx_{2}=0,
$$
Similarly, we can show that
$$
\iint_{\Omega\cap\{\hat{\psi}\leq0\}}|\nabla\hat{\psi}|^{2}dx_{1}dx_{2}=0,
$$
which implies that $\hat{\psi}=0$.
This completes the proof.
\end{proof}

\section{Refined Properties of Stream Functions for {\bf Problem 1}}

In this section, we derive some refined properties for solutions to
the boundary value problem \eqref{2.33}, {\it Problem 2}.
More precisely, it is
shown that $\psi_{x_{2}}$ is always positive,
together with the asymptotic behavior and the estimates obtained in
\S 3--5, yields that $(\rho,u,v,p)$ induced by $\psi$ satisfies the
original Euler equations, the boundary conditions, the constrains on
the mass flux, the Bernoulli constant, and the entropy equation.

\noindent
\begin{lemma}\label{5} Let the boundary $\partial\Omega$ satisfies \eqref{1.3}--\eqref{1.6}.
Then there exists $\delta_{4}\in(0,\bar{\delta}_{0}]$ such that, if
\begin{itemize}
\item[(i)]\
$\|(S-\underline{S},B-\underline{B})\|_{C^{1,1}([0,1])}\leq\delta$ with $0<\delta\leq \delta_{4}$,

\medskip
\item[(ii)]\ $m\in(\delta^{\beta},\bar{m})$,

\medskip
\item[(iii)] \ $\psi$ satisfies \eqref{4.23} and solves {\rm Problem 2},
\end{itemize}
then $\psi$ satisfies
\be\label{5.1}
0<\psi<m\qquad \text{in}\ \Omega,
\ee
and
\be\label{5.2}
\psi_{x_{2}}>0\qquad \text{in}\ \bar{\Omega}.
\ee
\end{lemma}

\begin{proof} We divide the proof into four steps.

\medskip
1. {\it Equation and the boundary condition.}
 From \eqref{4.23} and the boundary conditions: $\psi(0)=0$ and $\psi(1)=m$, we have
\be\label{5.3}
\psi_{x_{2}}\geq0\qquad \text{on}\ \partial\Omega.
\ee
Let $w=\psi_{x_{2}}$. From Lemma \ref{3}, $w$ satisfies
\begin{eqnarray}\label{6.4a}
&&\partial_{i}\Big(\frac{A_{ij}(\nabla\psi,\psi)}{\rho^{2}(|\nabla\psi|^{2},\psi)}\partial_{j}w\Big)
-\partial_{i}\Big(\frac{\rho_{2}(|\nabla\psi|^{2},\psi)\partial_{i}\psi}
{\rho^{2}(|\nabla{\psi}|^{2},{\psi})}w\Big)\nonumber\\[2mm]
&&\quad =\Theta(|\nabla\psi|^{2},\psi) w
+\vartheta(|\nabla\psi|^{2},\psi)\partial_{i}\psi\partial_{i}w
\end{eqnarray}
in the weak sense, where $A_{ij}$, $\Theta$, and $\vartheta$ are
defined in \eqref{4.7}, except replacing
$(\tilde{\mathcal{S}},\tilde{\mathcal{B}},\tilde{\rho})$ by
$(\mathcal{S}, \mathcal{B}, \rho)$.

\medskip
2. {\it Positivity in $\Omega$:} That is,
\be\label{5.5}
w\geq 0\qquad \text{in}\,\,\, \Omega.
\ee
First, the far field behavior of $\psi$ implies that
$\psi_{x_2}\ra \rho_i u_i>0$ when $(-1)^{i+1}x_1\to\infty, i=0,1$, which implies that
$w(x_{1},x_{2})>0$ for $|x_{1}|>L$ with $L$ sufficiently large. As
before, multiplying \eqref{6.4a} by $w^{-}=\min\{w,0\}$, integrating
it on both sides and noticing \eqref{5.3}, we have
\begin{eqnarray*}
&&\iint_{\{U\leq0\}}\frac{|\nabla w|}{\rho^{2}(|\nabla\psi|^{2},\psi)}dx_{1}dx_{2}\\
&&\leq -\iint_{\{w\leq0\}}\big(\mathcal{B}''(\psi)
-\frac{1}{\gamma}\mathcal{S}''(\psi)
\rho^{\gamma-1}(|\nabla\psi|^{2},\psi) \big)w^{2}dx_{1}dx_{2}\\
&&\leq  C\delta\iint_{\{U\leq0\}} w^{2} dx_{1}dx_{2}.
\end{eqnarray*}
For each $x_1$, we define an open set:
$$
K_{x_{1}}:=\{x_{2}\, :\, f(x_{1})\leq x_{2}\leq f(x_{2}),\
w(x_{1},x_{2})<0\}=\cup_iI_{x_{1}}^i,
$$
where each $I^{i}_{x_{1}}$ is a connected open component of $K_{x_{1}}$.
Then for every $x_{2}\in I^{i}_{x_{1}}$,
$$
w(x_{1},x_{2})=\int_{\min I^{i}_{x_{1}}}^{x_{2}}\partial_{x_{2}}w(x_{1},s)ds.
$$
Therefore,
\begin{eqnarray*}
\iint_{\{w\leq0\}} w^{2} dx_{1}dx_{2}
&=&\int_{-l}^{l}dx_{1}\sum\limits_{i}\int_{I^{i}}w^{2}(x_{1},x_{2})dx_{1}dx_{2}\\
&=&\int_{-l}^{l}dx_{1}\sum\limits_{i}\int_{ I^{i}}(\int_{\min I^{i}_{x_{1}}}^{x_{2}}\partial_{x_{2}}w(x_{1},s)ds)^{2}dx_{2}\\
&\leq&\int_{-l}^{l}\sum\limits_{i}\int_{ I^{i}}\int_{\min I^{i}}^{\max I^{i}}(\partial_{x_{2}}w(x_{1},s))^{2}ds
(\max I^{i}-\min I^{i})^{2}dx_{2}\\
&\leq&\max\limits_{x_{1}\in\mathcal{R}}|f(x_{2})-f(x_{1})|^{2}\iint_{\{w\leq0\}} |\nabla w|^{2} dx_{1}dx_{2}.
\end{eqnarray*}
Hence,
$$
\iint_{\{w\leq0\}}\frac{|\nabla w|}{\rho^{2}(|\nabla\psi|^{2},\psi)}dx_{1}dx_{2}
\leq  C\delta\iint_{\{w\leq0\}} |\nabla w|^{2} dx_{1}dx_{2},
$$
which means
$$
\iint_{\{w\leq0\}} |\nabla w|^{2} dx_{1}dx_{2}\leq0,
$$
Thus, \eqref{5.5} must hold.

\medskip
3. {\it Strict positivity in $\Omega$:} That is,
\be\label{5.6}
\psi_{x_{2}}=w> 0\qquad \text{in}\ \Omega
\ee
for any weak solutions $w$ to \eqref{6.4a}. Denote
$\tilde{w}:=e^{-\sigma x_{2}}w$, which is a nonnegative weak solution
to
$$
\partial_{i}\big(\frac{A_{ij}}{\rho^{2}}e^{\sigma x_{2}}\partial_{j}\tilde{w}\big)
+\big(\frac{A_{i2}}{\rho^{2}}\sigma-\frac{\rho_{2}\partial_{i}\psi}{\rho^{2}}-
\vartheta(|\nabla\psi|^{2},\psi)\partial_{i}\psi\big)e^{\sigma x_{2}}\partial_{i}\tilde{w}
+Ge^{\sigma x_{2}}\tilde{w}=0,
$$
where $A_{ij}$ and $\vartheta$ are defined in \eqref{4.7}, and
$$
G=\frac{A_{22}}{\rho^{2}}\sigma^{2}
+\Big(\partial_{i}(\frac{A_{i2}}{\rho^{2}})-\frac{\rho_{2}\partial_{2}\psi}{\rho^{2}}
-\vartheta(|\nabla\psi|^{2},\psi)\partial_{2}\psi\Big)\sigma
-\partial_{i}\Big(\frac{\rho_{2}\partial_{i}\psi}{\rho^{2}}\Big)
-\Theta(|\nabla\psi|^{2},\psi)
$$
with $\Theta$ also defined in \eqref{4.7}. Choosing $\sigma>0$
sufficiently large so that $G>0$, then
$$
\partial_{i}\Big(\frac{A_{ij}}{H^{2}}e^{\sigma x_{2}}\partial_{j}\tilde{w}\Big)
+\Big(\frac{A_{i2}}{\rho^{2}}\sigma-\frac{\rho_{2}\partial_{i}\psi}{\rho^{2}}-
\vartheta(|\nabla\psi|^{2},\psi)\partial_{i}\psi\Big)e^{\sigma
x_{2}}\partial_{i}\tilde{w} \leq0.
$$
This implies that \eqref{5.6} holds, so does inequality \eqref{5.1}.

\medskip
4. {\it Positivity on boundary.} We now show that
$\psi_{x_2}>0$ at $W_1\cup W_2$ in this step. Without loss of generality, we
prove it on $W_2$.

First, if $(\mathcal{S}\mathcal{B}^{-\gamma})'(m)<0$, since $\psi=m$
on $W_{2}$, then, for any $(x^{0}_{1},f_{2}(x^{0}_{1}))\in W_{2}$,
there exists a small disk $\mathcal{N}\subset \Omega$ satisfying
$\overline{\mathcal{N}}\cap\overline{\Omega}=(x^{0}_{1},f_{2}(x^{0}_{1}))$
such that
$\frac{d\ln(\mathcal{S}^{-\gamma}\mathcal{B})}{d\psi}\geq0$ in
$\mathcal{N}$, which implies
$$
A_{ij}(\nabla\psi,\psi)\partial_{ij}\psi>0 \qquad\mbox{ in $\mathcal{N}$}.
$$
Combining this with the fact that $\psi<m$, by the Hopf lemma, we have
$$
\psi_{x_{2}}(x^{0}_{1},f_{2}(x^{0}_{1}))>0.
$$

The remaining case is $(\mathcal{S}\mathcal{B}^{-\gamma})'(m)=0$. It
is easy to see that  $\psi$ satisfies
$$
A_{ij}(\nabla\psi,\psi)\partial_{ij}(\psi-m)-\rho_2|\nabla(\psi-m)|^2+R(\psi-m)=0
$$
with
$R=-\frac{\rho^2(\mathcal{B}'\rho-\frac{1}{\gamma}S'\rho^{\gamma-1})}{\psi-m}$.
By the Hopf lemma again, we have
$$
\partial_{x_2}\psi>0\qquad\text{in}\ W_2.
$$
Similarly, we can show that $\psi_{x_{2}}>0$ on $W_1$.
This completes the proof.
\end{proof}

\section{Proof of Theorem 1.1 Except the Critical Mass Flux}

We now prove Theorem 1.1 (Main Theorem of this paper), except the existence
part of the critical mass flux which will be
shown in Section 8 below.

\medskip
Let
$\delta_{0}:=\min\{\delta_{1},\delta_{2},\delta_{3},\delta_{4}\}>0$.
If
$\|(S-\underline{S},B-\underline{B})\|_{C^{1,1}([0,1])}\leq\delta$
with $0<\delta\leq \delta_{0}$, for any
$m\in(\delta^{\beta},2\delta^{\beta/2}_{0})$, there exists a
solution of {\it Problem 2}. It follows from Lemmas \ref{3} and
\ref{5} that the flow field induced by $\psi$ satisfies
\eqref{2.10}, and hence Proposition \ref{1} guarantees the existence
of Euler flows. Furthermore, {\it Propositions \ref{1} and \ref{6}} imply
the uniqueness of Euler flows with asymptotic \eqref{1.10} and
\eqref{1.12}, the mass flux condition \eqref{1.8}, and the
asymptotic behavior determined by \eqref{1.18}--\eqref{1.24}.

This completes the proof.

\section{Existence of the Critical Mass Flux}

In \S 5--7, we have shown that, for the given Bernoulli function and the entropy function
in the upstream satisfying \eqref{1.13}--\eqref{1.16a},
there exists a Euler flow, as long as $m\in(\delta^{\beta},2\delta^{\beta/2}_{0})$.
In this section, we  find the critical mass flux, which can be obtained by following the arguments
as in \cite{Bers1,Bers2,xx1}. For self-containedness, we give the proof in this section.

\noindent
\begin{proposition} Let the boundary $\partial\Omega$ satisfy \eqref{1.3}--\eqref{1.6},
$S(x_2)$ and $B(x_2)$ satisfy the asymptotic condition \eqref{1.10} and \eqref{1.12} for $x_2\in[0,1]$ respectively,
and let \eqref{1.16a} hold. Then there exists $\hat{m}\leq\bar{m}$ such that,
if $m\in(\delta^{\beta},\hat{m})$, there exits a unique $\psi$ of {\rm Problem 2} satisfying
\begin{eqnarray}
&&0<\psi<m\qquad \text{in}\ \Omega,\\
&& M(m):=\sup\limits_{\bar{\Omega}}\big\{|\nabla\psi|^{2}-(\gamma-1)\mathcal{S}(\psi)\rho^{\gamma+1}\big\}<0,
\end{eqnarray}
where
$\mathcal{B}(\psi)=\frac{u_0^{2}(\psi)}{2}+\mathcal{S}(\psi)\rho_0^{\gamma-1}(\psi)$.
Furthermore, either $M(m)\rightarrow0$ as $m\rightarrow\hat{m}$  or
there does not exist $\sigma>0$ such that \eqref{2.33} has solutions
for all $m\in(\hat{m},\hat{m}+\sigma)$ and \be\label{6.2}
\sup\limits_{m\in(\hat{m},\hat{m}+\sigma)}M(m)<0. \ee
\end{proposition}

\begin{proof}
For the given entropy function $S$ and Bernoulli function $B$ in the upstream
satisfying \eqref{1.10} and \eqref{1.12} and any $m\in(\delta^{\beta},\bar{m})$,
one can define $\rho_{0}$ and $u_{0}(x_{2})$ as in \S 2.
Note that $\rho_{0}$ and $u_0$ depend on $m$ by definition; thus in this section
we  denote them by $\rho_{0}(m)$ and $u_0(\psi;m)$, respectively.

Let $\{\varepsilon_{n}\}^{\infty}_{n=1}$ be a strictly decreasing sequence of positive numbers
such that $\varepsilon_{1}\leq\varepsilon_{0}/4$ and $\varepsilon_{n}\downarrow0$.
We introduce
$$
\zeta_{n}(s)=
\begin{cases}
s \qquad &\text{if}\ s<-2\varepsilon_{n},\\
-\frac{3}{2}\epsilon_{n}\qquad &\text{if}\ s\geq-\varepsilon_{n}.
\end{cases}
$$
Then  $\zeta_n$ is an increasing smooth function.
We define
\begin{eqnarray*}
&&\tilde{\triangle}_{n}(|\nabla\psi|^{2},\psi;m)\\[1.5mm]
&&:=
\zeta_{n}(|\nabla\psi|^{2}-(\gamma-1)\tilde{\mathcal{S}}(\psi)\tilde{\rho}^{\gamma+1}(|\nabla\psi|,\psi;m))
+(\gamma-1)\tilde{\mathcal{S}}(\psi)\tilde{\rho}^{\gamma+1}(|\nabla\psi|,\psi;m).
\end{eqnarray*}
Then there exist two positive constants $\lambda(n)$ and $\Lambda(n)$ such that
$$
\lambda(n)|\xi|^{2}\leq \tilde{A}^{n}_{ij}(q,z;m)\xi_{i}\xi_{j}\leq\Lambda(n) |\xi|^{2}
$$
for any $z\in\mathbb{R}$, $q\in\mathbb{R}^{2}$, and $\xi\in\mathbb{R}^{2}$, where
$$ \tilde{A}^{(n)}_{ij}(q,z;m)=\tilde{\rho}^{n}(|q|^{2},z;m)\delta_{ij}-2\tilde{\rho}^{n}_{1}(|q|^{2},z;m)\xi_{i}\xi_{j}.
$$
Thus, for any $m\in(\delta^{\beta},\bar{m})$, there exists a solution $\psi^{n}(x;m)$ to the problem:
\bee\label{6.5}
\begin{cases}
 \tilde{A}^{(n)}_{ij}(q,z;m)\partial_{ij}\psi
=\mathcal{F}_{n}(\nabla\psi,\psi;m)\qquad  &\text{in}\ \Omega,\\[2mm]
\psi=\frac{x_{2}-f_{1}(x_{1})}{f_{2}(x_{1})-f_{1}(x_{1})}m\qquad &\text{on}\ \partial\Omega,
\end{cases}
\eee
where
$$
\mathcal{F}_{n}(\nabla\psi,\psi;m)=\big(\tilde{B}'
-\frac{1}{\gamma}\tilde{\mathcal{S}}'(\tilde{\rho}^{n})^{\gamma-1}\big)(\tilde{\rho}^{n})^2
+\tilde{\rho}^{(n)}_{2}|\nabla\psi|^2.
$$
Moreover, if
\be\label{6.6}
|\nabla\psi^{(n)}|^{2}-(\gamma-1)\tilde{\mathcal{S}}(\psi)\tilde{\rho}^{\gamma+1}
(|\nabla\psi|,\psi;m)\leq-2\varepsilon_{n},
\ee
then $\zeta'_{n}=1$. Similar to \S 3, we have
$$
0\leq\psi^{n}(x;m)\leq m.
$$
By the definition of  $\tilde{\mathcal{S}}$ and $\tilde{\mathcal{B}}$,
we  can estimate $I_{5}$ in \eqref{4.13}, which is independent of $\epsilon_{n}$.
Furthermore, it follows from  Lemma \ref{3} that the solution in \eqref{6.5} satisfying \eqref{6.6}
has the far field behavior as \eqref{4.6}.
In addition, by Proposition \ref{6}, such a solution is unique
in the class of solutions satisfying \eqref{4.6}.

Note that, in general, we do not know the uniqueness of solutions to
the boundary value problem \eqref{6.5}.
Set
\be
S_{n}(m)=\{\psi^{n}(x;m)\,:\, \psi^{n}(x;m)\ \text{solves problem \eqref{6.5}}\}.
\ee
Define
\be
M_{n}(m)=\inf\limits_{\psi^{n}\in S_{n}(m)}\sup\limits_{\bar{\Omega}}\big\{|\nabla\psi^{(n)}|^{2}
-(\gamma-1)\tilde{\mathcal{S}}(\psi)\tilde{\rho}^{\gamma+1}_n(|\nabla\psi|,\psi;m)\big\},
\ee
and
$$
T_{n}=\{s\, :\,\delta^{\beta}\leq s,\ M_{n}(m)\leq-4\varepsilon_{n}\ \text{if}\ m\in(\delta^{\beta},s)\}.
$$
It follow from Proposition \ref{1}, Lemma \ref{3}, and Proposition \ref{4} that
$$
[\delta_{0}^{\beta},2\delta_{0}^{\beta/2}]\subset T_{n},
$$
hence $T_{n}\neq\emptyset$. We define $m_{n}=\sup T_{n}$.

The sequence $\{m_{n}\}$ has the following properties:

1. $M_{n}(m)$ is left continuous for $m\in(\delta^{\beta},m_{n}]$.
Indeed, let $\{m^{(k)}_{n}\in(\delta^{\beta},m_{n})\}$ and $m^{(k)}_{n}\uparrow m$.
Since $M_n(m^{(k)}_{n})\leq-4\varepsilon_{n}$, we have
$$
\|\psi^{(n)}(x;m^{(k)}_{n})\|_{C^{2,\alpha}}\leq C(n).
$$
Therefore, there exists a subsequence $\psi^{(n)}(x;m^{(k_{l})}_{n})$ such that
$$
\psi^{(n)}(x;m^{(k_{l})}_{n})\rightarrow\psi.
$$
Moreover, $\psi$ solves \eqref{6.5}, and $M_{n}(m)\leq\lim\psi^{(n)}(x;m^{(k_{l})}_{n})$.
Thus,
$$
M_{n}(m)\leq-4\varepsilon_{n}.
$$
Note that all these solutions satisfy the far field behavior as \eqref{4.6},
by uniqueness of solutions in this class,
$$
M_{n}(m)=\lim\psi^{(n)}(x;m^{(k)}_{n}).
$$

2. $m_{n}\leq\bar{m}$: If this were not true, by the definition of $m_{n}$,
$\bar{m}\in T_{n}$. It follows from the left-continuity of $M_{n}(m)$ that
$$
M_{n}(\bar{m})\leq-4\varepsilon_{n}.
$$
Thus, by means of the proof of Lemma \ref{3}, $\psi^{n}(x;\bar{m})$ has far field behavior as in \eqref{4.6}.
However, it follows from the definition of $\bar{m}$ that
\begin{eqnarray*}
&&\sup\limits_{x\in\bar{\Omega}}\big\{|\nabla\psi^{(n)}|^{2}-
(\gamma-1)\tilde{\mathcal{S}}(\psi)\tilde{\rho}^{\gamma+1}_n(|\nabla\psi^{(n)}|,\psi^{(n)};m)\big\}\\
&&\geq \sup\limits_{s\in[0,1]}\max
 \left\{
 \begin{array}{ll}
 |\rho_{0}(s;\bar{m})u_{0}(s;\bar{m})|^{2}-(\gamma-1)S(s)\rho_0^{\gamma+1}(s;\bar{m}),\\[2mm]
 |\rho_{1}(s;\bar{m})u_{1}(s;\bar{m})|^{2}-(\gamma-1)S(s)\rho_1^{\gamma+1}(s;\bar{m})
 \end{array}
 \right\}\\
&&\geq \sup\limits_{s\in[0,1]}\max
 \left\{
 \begin{array}{ll}
 2B(s)\rho_0^2(s;\bar{m})-(\gamma+1)S(s)\rho_0^{\gamma+1}(s;\bar{m}),\\[2mm]
 2B(s)\rho_1^2(s;\bar{m})-(\gamma+1)S(s)\rho_1^{\gamma+1}(s;\bar{m})
 \end{array}
 \right\}\\
&&\geq
\sup\limits_{s\in[0,1]}\varrho^2(\bar{D};s)S^{\frac{1}{\gamma}}(s)
  \Big(2D(x_2)-(\gamma+1)\big(\frac{\gamma\mathfrak{p}(\bar{D})}{\gamma-1}\big)^{\frac{\gamma-1}{\gamma}}\Big)\\
&&\geq
\sup\limits_{s\in[0,1]}2\varrho^2(\bar{D};s)S^{\frac{1}{\gamma}}(s)
  \big(D(x_2)-\bar{D}\big)\\
&&=0,
\end{eqnarray*}
where $\rho_1(s;\bar{m})=\rho_1(y(s);\bar{m})$,
$u_1(s;\bar{m})=u_1(y(s);\bar{m})$, and $y(s)$ is the function
defined in \eqref{2.25}. Thus, $M_{n}(\bar{m})\geq0$. This is a
contradiction. Therefore, $m_{n}\leq\bar{m}$.

Finally,  $\{m_{n}\}$ is an increasing sequence,
which follows from the definition of $\{m_{n}\}$ directly.
Define
$$
\hat{m}=\lim\limits_{n\rightarrow\infty }m_{n}.
$$
Then $\hat{m}$ is well-defined and $\hat{m}\leq\bar{m}$.
Note that, for any $m\in(\delta^{\beta},\hat{m})$, there exists $m_{n}>m$ such that
$M_{n}(m)\leq-4\varepsilon_{n}$.
Thus,
$$
\phi=\psi^{(n)}(x;m)
$$
solves \eqref{2.33} and
$$
\sup\limits_{\bar{\Omega}}\big\{|\nabla\psi|^{2}
-(\gamma-1)\mathcal{S}(\psi)\tilde{\rho}^{\gamma+1}_n(|\nabla\psi^{(n)}|,\psi^{(n)};m)\big\}
=M_{n}(m)\leq-4\varepsilon_{n}.
$$
If $\sup\limits_{m\in(\delta^{\beta},\hat{m})}M_{n}(m)<0$, then there exists $n$ such that
$$
\sup\limits_{m\in(\delta^{\beta},\hat{m})}M_{n}(m)<-4\varepsilon_{n}.
$$
Then the same argument as the proof for the left continuity of $M_{n}(m)$
 on $(\delta^{\beta},m_{n}]$ yields
$$
M_{n}(\hat{m})<-4\varepsilon_{n}.
$$
Suppose that there exists $\sigma>0$ such that \eqref{2.33} always
has a solution $\psi$ for $m\in(\hat{m},\hat{m}+\sigma)$, and
\be
\sup\limits_{m\in(\hat{m},\hat{m}+\sigma)}M(m)<0.
\ee
Then there exists $k>0$ such that
$$
\sup\limits_{m\in(\hat{m},\hat{m}+\sigma)}M(m)=\sup\limits_{m\in(\hat{m},\hat{m}+\sigma)}
\big\{|\nabla\psi|^{2}-(\gamma-1)\mathcal{S}(\psi)\tilde{\rho}^{\gamma+1}(|\nabla\psi|,\psi;m)\big\}
<-4\varepsilon_{n+k}.
$$
This yields that $m_{n+k}\geq\hat{m}+\sigma$, which is a contradiction.
Thus, either $M(m)\rightarrow0$ or there does not exist $\sigma>0$ such that \eqref{2.33} has a solution
for all $m\in(\hat{m},\hat{m}+\sigma)$ and \eqref{6.2} holds.
This completes the proof.
\end{proof}

\bigskip
\textbf{Acknowledgements}.
The authors thank Myoungjean Bae for helpful discussions.
The research of
Gui-Qiang Chen was supported in part by the National Science
Foundation under Grants
DMS-0935967 and DMS-0807551, the UK EPSRC Science and Innovation
Award to the Oxford Centre for Nonlinear PDE (EP/E035027/1),
the NSFC under a joint project Grant 10728101, and
the Royal Society--Wolfson Research Merit Award (UK).
The research of Xuemei Deng  was supported in part by
by China Scholarship Council No.
2008631071 and by the EPSRC Science and
Innovation Award to the Oxford Centre for Nonlinear PDE
(EP/E035027/1).
The research of Wei Xiang was supported in part by
China Scholarship Council No.
 2009610055 and by the EPSRC Science and
Innovation Award to the Oxford Centre for Nonlinear PDE
(EP/E035027/1).

\end{document}